\documentclass{article}

\usepackage{algorithm}
\usepackage{algorithmic}
\usepackage{amsmath}
\usepackage{amssymb}
\usepackage{latexsym}
\usepackage{amsthm}
\usepackage{ecltree}
\usepackage{mathrsfs}
\usepackage{mathtools}
\usepackage{xcolor}
\usepackage{comment}
\usepackage{tikz}
\usepackage{bm}
\usepackage{enumerate}
\usepackage[bbgreekl]{mathbbol}
\usetikzlibrary{trees}
\mathtoolsset{showonlyrefs=true}
\numberwithin{equation}{section}
\usepackage[margin=25truemm]{geometry}

\newtheorem{definition}{Definition}[section]
\newtheorem{proposition}[definition]{Proposition}
\newtheorem{lemma}[definition]{Lemma}
\newtheorem{theorem}[definition]{Theorem}
\newtheorem{corollary}[definition]{Corollary}
\newtheorem{remark}[definition]{Remark}
\newtheorem{example}[definition]{Example}

\newcommand{\NN}{{\mathscr{N}}}
\newcommand{\PP}{{\mathscr{P}}}
\newcommand{\es}{{\emptyset}}

\newcommand{\st}{{\ast}}

\newcommand{\down}{{\downarrow}}

\newcommand{\bin}{\mathcal{B}}
\newcommand{\uint}{\mathbb{Z}_{\geq 0}}
\newcommand{\imp}{\mathbb{Im}}
\newcommand{\short}{{\imp}^{\bin}}

\newcommand{\mono}{{\imp}^{\bin \down}}

\newcommand{\birth}{\mathrm{b}}
\newcommand{\grundy}{\mathcal{G}}
\newcommand{\type}{\mathrm{type}}
\DeclareMathOperator*{\mex}{mex}

\newcommand{\eqlab}[2]{\overset{(\mathrm{#1})}{#2}}
\newcommand{\defeq}[1]{\overset{\mathrm{def}}{#1}}

\begin{document}
\title{Impartial Games with Activeness}  
\author{Kengo Hashimoto}

\date{University of Fukui, E-mail: khasimot@u-fukui.ac.jp}

\maketitle

\begin{abstract}
A combinatorial game is a two-player game without hidden information or chance elements.
The main object of combinatorial game theory is to obtain the outcome, which player has a winning strategy, of a given combinatorial game.
Positions of many well-known combinatorial games are naturally decomposed into a disjunctive sum of multiple components and can be analyzed independently for each component.
Therefore, the study of disjunctive sums is a major topic in combinatorial game theory.
Combinatorial games in which both players have the same set of possible moves for every position are called impartial games.
In the normal-play convention, it is known that the outcome of a disjunctive sum of impartial games can be obtained by computing the Grundy number of each term.
The theory of impartial games is generalized in various forms.
This paper proposes another generalization of impartial games to a new framework, impartial games with activeness: each game is assigned a status of either ``active'' or ``inactive''; the status may change by moves;
a disjunctive sum of games ends immediately, not only when no further moves can be made, but also when all terms become inactive.
We formally introduce impartial games with activeness and investigate their fundamental properties.
\end{abstract}

\section{Introduction}

\subsection{Background}
A \emph{combinatorial game} is a two-player game without hidden information or chance elements; e.g., Chess, Go, and Checkers.
In a combinatorial game, two players take turns to make a move alternately.
If it is guaranteed that a winner is determined within a finite number of moves,
then exactly one of the two players has a winning strategy.
The main object of combinatorial game theory is to obtain the \emph{outcome $o(G)$}, which player has a winning strategy, of a given combinatorial game $G$.

As a classical result in combinatorial game theory, Bouton's theorem \cite{Bou1901} on Nim is well-known.
Nim is a combinatorial game played according to the following rule.
\begin{quote}
Fix a positive integer $n$ and non-negative integers $a_1, a_2, \ldots, a_n$.
Initially, there are $n$ piles of stones, and 
the $i$-th pile contains $a_i$ stones for each $i \in [n]$.
On each turn, a player removes one or more stones from any single pile.
The player who cannot move on the player's turn because all piles are empty loses, and the other player wins.
\end{quote}
The outcome of a given position of Nim can be obtained by the next proposition.

\begin{proposition}[Bouton \cite{Bou1901}]
\label{prop:nim}
In Nim, the second player has a winning strategy if and only if $a_1 \oplus a_2 \oplus \cdots \oplus a_n = 0$,
where $\oplus$ denotes \emph{nim-sum} (a.k.a.~bitwise XOR), which is the binary addition without carrying.
\end{proposition}

The \emph{normal-play} (resp.~\emph{mis\`{e}re-play}) is the way to determine a winner in which a play continues until no further moves can be made, and the player who cannot move on the player's turn loses (resp.~wins), like Nim.
The normal-play is actively studied because of its rich structure,
and this paper also assumes the normal-play unless otherwise specified.

By convention, we refer to an individual game position as a \emph{game} rather than a system of playable rules, such as Chess and Go.
Instead, a system of playable rules is referred to as a \emph{ruleset}.

Positions in many well-known rulesets are naturally decomposed into a \emph{disjunctive sum} of multiple smaller positions:
the disjunctive sum $G + H$ of games $G$ and $H$ is defined as the game in which
$G$ and $H$ are played in parallel, and a player makes a move on exactly one of $G$ and $H$ in a turn.
For example, a position of Nim with $n$ piles is the disjunctive sum of $n$ positions with a single pile of Nim.
An equivalence relation among games is introduced by regarding two games $G$ and $H$ as equivalent if $o(G + X) = o(H + X)$ for any game $X$.
Namely, if $G$ and $H$ are equivalent, then a term $G$ in a disjunctive sum can be replaced with $H$ without changing the outcome.
One of the major approaches to analyzing a game is to break down a given game into a disjunctive sum of multiple components and replace each component with an equivalent and simpler game.

Games in which both players have the same set of possible moves, like Nim, are called \emph{impartial} games\footnote{Chess, Go, and Checkers are not impartial because a player can only move (or put) pieces of their assigned color.}.
It is known that every impartial game is equivalent to a single pile of Nim consisting of the same number of stones as its \emph{Grundy number} \cite{Gru39, Spr35}.
Therefore, without altering the outcome, a disjunctive sum of impartial games can be reduced to a position of Nim,
the outcome of which can be obtained by Proposition \ref{prop:nim}.

Various generalizations of the theory of impartial games have been proposed:
Smith \cite{Smi66} generalized Grundy numbers to loopy games, in which a play could go on forever;
Plambeck and Siegel \cite{Pla09, PS08} proposed a powerful method for analyzing impartial games in the mis\`{e}re-play convention using \emph{mis\`{e}re quotients}, which is a monoid formed by equivalence classes of impartial games.
Larsson, Nowakowski, and Santos \cite{LNS22} studied impartial games with \emph{entailing moves}, which
widely cover moves that force the opponent to play on a specific component in a disjunctive sum, and consecutive moves by a single player.

\subsection{Organization and Contribution}
This paper generalizes the theory of impartial games to a new framework, impartial games \emph{with activeness}.
To describe the main idea of them, we begin with examples as follows.

\begin{example}
\label{ex:nim-once}
Consider adding the following rule to Nim:
if every pile has been played at least once, then the game ends immediately (and the player who made the last move wins by the normal-play convention).

This rule can also be rephrased as follows.
Initially, each pile is ``active''.
Once a pile is played on, its state changes to ``inactive''.
When every pile is inactive, the whole game immediately ends, and the player who made the last move wins.
Note that as long as the whole game is not finished, it is also allowed to make a move on an inactive pile.
\end{example}

\begin{example}
Consider adding the following rule to Nim:
if every pile has an even number of stones, then the game ends immediately (and the player who made the last move wins by the normal-play convention).

This rule can also be rephrased as follows.
A pile with an odd (resp.~even) number of stones is active (resp.~inactive).
When every pile is inactive, the whole game immediately ends, and the player who made the last move wins.
Note that as long as the whole game is not finished,
it is also allowed to make a move on an inactive pile,
and the inactive pile may become active by the move.
\end{example}

As a generalized framework for the above, we introduce impartial games with \emph{activeness} as follows.
\begin{enumerate}[(i)]
\item Each game is assigned an attribute of either ``active'' or ``inactive''.
\item The game ends immediately not only when no further moves can be made, but also when it becomes inactive.
\item A disjunctive sum of games is defined to be active if and only if at least one term is active.
\end{enumerate}
By (ii) and (iii), a disjunctive sum of games ends immediately, not only when no further moves can be made (i.e., no further moves can be made on any term),
but also when the sum becomes inactive (i.e., all terms in the sum become inactive).
Note that it is allowed to make a move on an inactive term in the sum (except when the whole sum has ended because all terms in the sum are inactive);
as a result of the move, the inactive term may become active.

In this paper, we formalize games with activeness and investigate their fundamental properties.

This paper is organized as follows.
\begin{itemize}
\item In Section \ref{sec:preliminaries}, we briefly review known results on impartial games without activeness.

\item In Section \ref{sec:main}, we present the main results of this paper.
\begin{itemize}
\item In Subsection \ref{subsec:activeness}, 
we formalize games with activeness and define the set $\short$ of all games with activeness.
We also define the outcome $o(G^g)$ of $G^g \in \short$, a disjunctive sum $G^g + H^h$ of two games $G^g, H^h \in \short$, and 
an equivalence relation $=$ of $\short$ as
\begin{align}
G^g = H^h \defeq\iff \forall X^x \in \short,\,\, o(G^g+X^x) = o(H^h+X^x)
\end{align}
as an analogue of the definitions for games without activeness.
Furthermore, we introduce corresponding concepts to nimbers and Grundy numbers.
Throughout the subsection, we also confirm fundamental properties of the introduced concepts.

\item In Subsection \ref{subsec:distinguish}, we focus on games that belong to pairwise distinct equivalent classes modulo $=$.
We prove Theorem \ref{thm:distinguish} (\emph{Distinguishing Theorem}), which gives a necessary and sufficient condition for the following condition:
there exists an active game $X^1 \in \short$ such that
$o(G_1^{g_1} + X^1) = o(G_2^{g_2} + X^1) = \cdots = o(G_m^{g_m} + X^1) =\PP$ and
$o(H_1^{h_1} + X^1) = o(H_2^{h_2} + X^1) = \cdots = o(H_n^{h_n} + X^1) =\NN$
for given games $G^{g_1}_1, G^{g_2}_2, \ldots, G^{g_m}_m$ and $H^{h_1}_1, H^{h_2}_2, \ldots, H^{h_n}_n$.

\item In Subsection \ref{subsec:simplify}, we prove Theorem \ref{prop:simplify} (\emph{Simplifying Theorem}), which simplifies a given game to an equivalent and simpler game under a certain condition.
More precisely, it gives a sufficient and necessary condition of $G^g = H^h$, where $G^g$ is \emph{covered} (Definition \ref{def:cover}) by $H^h$ .

\item In Subsection \ref{subsec:canonical}, we define a unique ``canonical'' game as a representative of  each equivalence class.
More precisely, we say that a game is \emph{canonical} if none of its subpositions (Definition \ref{def:subpos}) have a reversible (Definition \ref{def:reversible}) option,
and we prove that there exists exactly one canonical game for each equivalence class as Theorem \ref{thm:canonical} (\emph{Canonical Form Theorem}).
We also show that \emph{transitive} (Definition \ref{def:transitive}) games, including nimbers, are canonical,
and there are also an infinite number of non-transitive canonical games.

\item In Subsection \ref{subsec:nimber}, we discuss nimbers as a special case of the results up to Subsection \ref{subsec:canonical}.
In particular, we prove the mex rule and the addition rule among nimbers,
and further demonstrate that simple rules apply under more specific conditions.
\end{itemize}

\item In Section \ref{sec:mono}, we introduce the set $\mono$ of \emph{non-increasing} games, in which once a game becomes inactive, it never returns to the active status.
We confirm that even when limiting the discussion to non-increasing games, the results in Section \ref{sec:main} hold almost unchanged.
\end{itemize}

Part of the argument in Subsections \ref{subsec:distinguish}--\ref{subsec:canonical} is an analogue of the theory of mis\`{e}re canonical form \cite[Section V-3]{Sie13}.

Hereafter, all games are impartial even if not explicitly stated.
Ordinary (impartial) games and (impartial) games with activeness are distinguished by being called
(impartial) games \emph{without} activeness and (impartial) games \emph{with} activeness, respectively.

\section{Impartial Games}
\label{sec:preliminaries}

In this section, we briefly introduce known results on impartial games without activeness.
We leave the proofs of the propositions and other detailed discussions to the textbooks such as \cite{ANW19, BCG18, Con00, HG16, Sie13}.

Let $\uint$ denote the set of all non-negative integers. For $n \in \uint$, let $[n]$ denote the set of all positive integers at most $n$.
Namely,
\begin{align}
\uint \defeq= \{0, 1, 2, \ldots, \}, \quad [n] \defeq= \{1, 2, \ldots, n\}.
\end{align}

As a player makes a move in game $G$, the game $G$ transitions to another game $G'$ closer to the end than $G$.
An impartial game $G$ is identified by the set of all possible transitions by a move from $G$.
More precisely, an element $G$ of the set $\imp_n$ of all impartial games which terminate within $n$ moves is identified by
the set of all impartial games reached by a move from $G$, which is a subset of the set $\imp_{n-1}$ of all games which terminate within $n-1$ moves.
The set $\imp$ of all impartial games is the union of the sets $\imp_n$ for all $n \in \uint$.
The formal definition is as follows.

\begin{definition}
\label{def:impartial}
For $n \geq \uint$, the set $\imp_n$ is defined as
\begin{align}
\imp_n \defeq{=}
\begin{cases}
\{\es\} &\,\,\text{if}\,\,n = 0,\\
2^{\imp_{n-1}} &\,\,\text{if}\,\,n \geq 1,\\
\end{cases}
\end{align}
where $2^A$ denotes the power set of a set $A$.
A \emph{game} is an element of
\begin{align}
\imp \defeq{=} \bigcup_{n \in \uint} \imp_n.
\end{align}
\end{definition}

Note that $\imp_0 \subsetneq \imp_1 \subsetneq \imp_2 \subsetneq \cdots$ by the definition.
The empty game $\es$ corresponds to the end position in which neither player can make any more moves.
An element of a game $G$ is called an \emph{option} of $G$, which corresponds to a possible transition of the position by a move.
If $G, H \in \imp$ are the same elements of $\imp$, then we write $G \cong H$ instead of $G = H$.
A notation $G = H$ is reserved for another meaning as stated later, and $G = H$ and $G \cong H$ are strictly distinguished.

The games in the next definition, called \emph{nimbers}, are fundamental.
\begin{definition}
\label{def:imp-star}
For $n \in \uint$, a game $\st n$ is defined as
\begin{align}
\st n \defeq\cong \{\st n' : n' \in \uint, n' < n\} \label{eq:y2wjh4pe0bmr}
\end{align}
recursively. In particular, $\st 0 \defeq\cong \es$.
\end{definition}
Note that $\st n$ corresponds to a single pile consisting of $n$ stones in Nim.

\begin{example}
\begin{align}
\imp_0 &= \{\es\} = \{\st 0\},\\
\imp_1 &= 2^{\imp_0} = \{\es, \{\es\}\} = \{\st 0, \st 1\},\\
\imp_2 &= 2^{\imp_1} = \{\es, \{\es\}, \{\{\es\}\}, \{\es, \{\es\}\}\} = \{\st 0, \st 1, \{\st 1\}, \st 2\}.
\end{align}
\end{example}

\subsection{The Outcome Classes}

For each game, exactly one player has a winning strategy.
Namely, one player can win by repeatedly ``making an appropriate response to the other player's move based on a certain strategy.''
The following mapping indicates which player has a winning strategy for a given game.

\begin{definition}
\label{def:imp-outcome}
We define a mapping $o \colon \imp \to \{\PP, \NN\}$ recursively as
\begin{align}
o(G) \defeq=
\begin{cases}
\NN &\,\,\text{if}\,\, \exists G' \in G \,\,\text{s.t.}\,\, o(G') = \PP,\\
\PP &\,\,\text{if}\,\, \forall G' \in G, o(G') = \NN.
\end{cases}
\end{align}
\end{definition}

The condition $o(G) = \NN$ indicates that the first player (i.e., the player who makes a move on $G$ next)  has a winning strategy in $G$, and $o(G) = \PP$ indicates that the second player does\footnote{The symbol $\NN$ (resp.~$\PP$) is derived from the initial letter of ``N''ext player (resp.~``P''revious player).}.
In fact, if $o(G) = \NN$, then the first player has the following winning strategy.
\begin{quote}
Call the two players Alice and Bob.
Suppose that Alice is the player who plays on $G$ first.
Since $o(G) = \NN$, Alice can make a move to some $G'$ with $o(G') = \PP$.
On the other hand, every option $G''$ of $G'$ satisfies $o(G'') = \NN$ by $o(G') = \PP$,
and thus Bob cannot avoid passing Alice a game $G''$ with $o(G'') = \NN$.
By repeating this process, Alice can ensure that the games she plays on her turn are always in $\NN$, allowing her to continue playing without exhausting her options.
Namely, Alice can guarantee that she will win.
\end{quote}

By abusing notation, we define $\PP \defeq{=} \{G \in \imp : o(G) = \PP\}$ and $\NN \defeq{=} \{G \in \imp : o(G) = \NN\}$.
Then $\imp$ is divided into a disjoint union $\imp = \NN \sqcup \PP$.
The sets $\PP$ and $\NN$ are called the \emph{outcome classes}, and $o(G)$ is called the \emph{outcome of $G$} for $G \in \imp$.

\begin{example}
Let $G \defeq\cong \{\{\es, \{\es\}\},\{\es, \{\{\es\}\}\}\}$.
We confirm $o(G) = \PP$ as follows.
\begin{enumerate}[(i)]
\item $o(\es) = \PP$ because $\PP$ has no options (in particular, every option is in $\NN$).
\item $o(\{\es\}) = \NN$ because it has an option $\es \in \PP$ by (i).
\item $o(\{\{\es\}\}) = \PP$ because the only option is $\{\es\} \in \NN$ by (ii).
\item $o(\{\es, \{\es\}\}) = \NN$ because it has an option $\es \in \PP$ by (i).
\item $o(\{\es, \{\{\es\}\}\}) = \NN$ because it has an option $\es \in \PP$ by (i).
\item $o(G) = o(\{\{\es, \{\es\}\},\{\es, \{\{\es\}\}\}\}) = \PP$
because every option is in $\NN$: $\{\es, \{\es\}\} \in \NN$ by (iv), and $\{\es, \{\{\es\}\}\} \in \NN$ by (v).
\end{enumerate}
\end{example}

\subsection{Addition of Games}

The (disjunctive) sum $G+H$ of games $G$ and $H$ is the game made by combining $G$ and $H$
in which $G$ and $H$ are played in parallel, and a player makes a move on exactly one of $G$ and $H$ in a turn.

\begin{definition}
\label{def:imp-sum}
For $G, H \in \imp$,
the \emph{(disjunctive) sum} $G + H$ of $G$ and $H$ is defined recursively as
\begin{align}
G+H &\defeq{=} \{G' + H : G' \in G\} \cup \{G + H' : H' \in H\}.
\end{align}
\end{definition}

\begin{example}
Consider the sum of $\{\es, \{\es\}\}$ and $\{\{\es\}\}$.
By the definition, we have
\begin{align}
\{\es, \{\es\}\} + \{\{\es\}\} \cong \{ \es + \{\{\es\}\}, \{\es\} + \{\{\es\}\}, \{\es, \{\es\}\} + \{\es\} \}.
\end{align}
Intuitively, $\{\es, \{\es\}\} + \{\{\es\}\}$ is the game in which a player makes a move on exactly one of $\{\es, \{\es\}\}$ and  $\{\{\es\}\}$ in a turn.
Namely, the player has the following three options:
\begin{itemize}
\item the move to $\es + \{\{\es\}\}$ playing on the first component $\{\es, \{\es\}\}$,
\item the move to $\{\es\} + \{\{\es\}\}$ playing on the first component $\{\es, \{\es\}\}$,
\item the move to $\{\es, \{\es\}\} + \{\es\}$ playing on the second component $\{\{\es\}\}$.
\end{itemize}
\end{example}

The addition is associative and commutative and has an identity $\es \in \imp$ as follows.

\begin{proposition}~
\label{prop:imp-monoid}
\begin{enumerate}[(i)]
\item For any $G, H, J \in \imp$, we have $(G + H) + J \cong G + (H + J)$.
\item For any $G, H \in \imp$, we have $G + H \cong H + G$.
\item For any $G \in \imp$, we have $G + \es \cong G$.
\end{enumerate}
\end{proposition}

By Proposition \ref{prop:imp-monoid} (i), we can write $(G+H)+J$ (equivalently, $G+(H+J)$) as $G+H+J$ without ambiguity.
We note that a game $\st a_1 + \st a_2 + \cdots + \st a_n$ represents a position of Nim in which there are $n$ piles, and the $i$-th pile consists of $a_i$ stones for $i \in [n]$.

\subsection{Equivalence among Games}

We next introduce a binary relation $=$ of $\imp$, which intuitively indicates that two games play the same role in a sum.

\begin{definition}
\label{def:imp-equiv}
We define a binary relation $=$ of $\imp$ as
\begin{align}
\{(G, H) \in \imp \times \imp : \forall X \in \imp, o(G+X) = o(H+X)\}.
\end{align}
Namely, 
\begin{align}
G = H \defeq\iff \forall X \in \imp,\,\, o(G+X) = o(H+X).
\end{align}
\end{definition}
Hence, if $G = H$, then $G$ in a sum can be interchanged with $H$ without changing the outcome of the whole sum.

The binary relation $=$ satisfies the following basic properties;
in particular, the binary relation $=$ is an equivalence relation of $\imp$.

\begin{proposition}~
\label{prop:order}
\begin{enumerate}[(i)]
\item For any $G, H, J \in \imp$, if $G = H$ and $H = J$, then $G = J$.
\item For any $G, H, J, K \in \imp$, if $G = H$ and $J = K$, then $G+J = H+K$.
\item For any $G, H \in \imp$, if $G = H$, then $o(G) = o(H)$.
\end{enumerate}
\end{proposition}

Replacing an option $G'$ of a game $G$ with another equivalent game $H'$ to $G'$ yields an equivalent game $H$ to $G$ as follows.
\begin{proposition}[Replacement Lemma]
\label{cor:replace}
For any $\{G'_1, G'_2, \ldots, G'_n\} \in \imp$ and $H'_1 \in \imp$, if $G'_1 = H'_1$, then we have $\{G'_1, G'_2, \ldots, G'_n\} = \{H'_1, G'_2, \ldots, G'_n\}$.
\end{proposition}

\subsection{Nimbers and Grundy Numbers}

If every option of a game $G$ is a nimber, then $G$ is equivalent to a nimber by the following proposition.

\begin{proposition}[Mex-Rule]
\label{prop:imp-mex}
For any finite set $A \subseteq \uint$, we have $\{\st a : a \in A\} = \st (\mex A)$,
where $\mex A \defeq= \min (\uint \setminus A)$.
\end{proposition}

By repeatedly applying Propositions \ref{cor:replace} and \ref{prop:imp-mex},
we can see that any game is equivalent to a nimber.

On the other hand, no two distinct nimbers are equivalent to each other as follows.
\begin{proposition}
For any $a, b \in \uint$, if $\st a = \st b$, then $a = b$.
\end{proposition}

Namely, the following proposition holds.
\begin{proposition}[Sprague\cite{Spr35}, Grundy\cite{Gru39}]
For any $G \in \imp$, there exists a unique integer $n \in \uint$ such that $G = \st n$.
\end{proposition}

\begin{example}
Let $G \defeq\cong \{\{\es, \{\es\}\},\{\es, \{\{\es\}\}\}\}$.
We confirm $G = \st 0 \cong \emptyset$ as follows.
\begin{enumerate}[(i)]
\item $\es \cong \st 0$ by \eqref{eq:y2wjh4pe0bmr}.
\item We have \begin{align}
\{\es\} \eqlab{A}\cong \{\st 0\} \eqlab{B}\cong \st 1,
\end{align}
where (A) follows from (i), and (B) follows from \eqref{eq:y2wjh4pe0bmr}.
\item We have \begin{align}
\{\{\es\}\} \eqlab{A}\cong \{\st 1\} \eqlab{B}= \st (\mex \{1\}) \cong \st 0,
\end{align}
where (A) follows from (ii), and (B) follows from Proposition \ref{prop:imp-mex}.
\item We have \begin{align}
\{\es, \{\es\}\} \eqlab{A}\cong \{\st 0, \st 1\} \eqlab{B}\cong \st 2,
\end{align}
where (A) follows from (i) and (ii), and (B) follows from \eqref{eq:y2wjh4pe0bmr}.
\item We have \begin{align}
\{\es, \{\{\es\}\}\} \eqlab{A}\cong \{\st 0, \{\{\es\}\}\} \eqlab{B}= \{\st 0, \st 0\} \cong \{\st 0\} \eqlab{C}\cong \st 1,
\end{align}
where (A) follows from (i), (B) follows from (iii) and Proposition \ref{cor:replace}, and (C) follows from \eqref{eq:y2wjh4pe0bmr}.
\item We have \begin{align}
G \cong \{\{\es, \{\es\}\},\{\es, \{\{\es\}\}\}\} \eqlab{A}= \{\st 2, \st 1\} \eqlab{B}= \st (\mex\{1, 2\}) \cong \st 0,
\end{align}
where (A) follows from (iv), (v), and Proposition \ref{cor:replace}, 
and (B) follows from Proposition \ref{prop:imp-mex}.
\end{enumerate}
\end{example}

Therefore, a sum of games can be simplified to a sum of nimbers.
Since a sum of nimbers corresponds to a position of Nim, we can obtain the outcome of it by Proposition \ref{prop:nim}.
More precisely, the following proposition holds and can be used to obtain the outcome of a sum of nimbers.

\begin{proposition}~
\label{prop:imp-star}
\begin{enumerate}[(i)]
\item For any $a \in \uint$, the following equivalence holds: $o(\st a) = \PP \iff a = 0$.
\item For any $a, b \in \uint$, we have $\st a + \st b = \st (a \oplus b)$.
\end{enumerate}
\end{proposition}

For $G \in \imp$, the unique integer $n \in \uint$ such that $G = \st n$ is called the \emph{Grundy number} of $G$ and denoted by $\grundy(G)$.
By Propositions \ref{cor:replace} and \ref{prop:imp-mex}, we have the following recursive formula to compute the Grundy number of a game $G$:
\begin{align}
\grundy(G) = \mex\{\grundy(G') : G' \in G\}. \label{eq:ofe48aki7qcy}
\end{align}

To summarize, given a sum of games $G \defeq\cong G_1 + G_2 + \cdots + G_n$,
we can obtain the outcome of $G$ by the following procedure.
\begin{enumerate}
\item For each $i \in [n]$, compute $\grundy(G_i)$ by \eqref{eq:ofe48aki7qcy} and replace $G_i$ with $\st \grundy(G_i)$.
\item Rewrite the sum $\st \grundy(G_1) + \st \grundy(G_2) + \cdots + \st \grundy(G_n)$
as $\st (\grundy(G_1) \oplus \grundy(G_2) \oplus \cdots \oplus \grundy(G_n))$ by applying Proposition \ref{prop:imp-star} (ii).
\item Determine $o(G) = \PP$ if and only if $\grundy(G_1) \oplus \grundy(G_2) \oplus \cdots \oplus \grundy(G_n) = 0$ by utilizing Proposition \ref{prop:imp-star} (i).
\end{enumerate}

Therefore, to obtain the outcome of a given sum of games,
one needs only to find the Grundy number of each term independently and check whether their nim-sum is zero.

\section{Main Results}
\label{sec:main}

In this section, we formally introduce games with activeness and investigate their properties as the main results of this paper.
We restate the main idea mentioned in the Introduction as follows.
\begin{enumerate}[(i)]
\item Each game is assigned an attribute of either ``active'' or ``inactive''.
\item The game ends immediately not only when no further moves can be made, but also when it becomes inactive.
\item A sum of games is defined to be active if and only if at least one term is active.
\end{enumerate}
In particular, a sum of games ends immediately not only when no further moves can be made (i.e., no further moves can be made on any term),
but also when the sum becomes inactive (i.e., all terms in the sum become inactive).
It is allowed to make a move on an inactive term in the sum (except when the whole sum has ended because all terms in the sum are inactive);
as a result of the move, the inactive term may become active.

\subsection{Impartial Games with Activeness}
\label{subsec:activeness}

We formalize the idea of games with activeness by modifying Definition \ref{def:impartial}.
A game with activeness is identified by the pair $(G, b)$ of
\begin{itemize}
\item the set $G$ of all possible transitions by a move from the game, and
\item an integer $b \in \bin \defeq= \{0, 1\}$ representing whether the game is active or inactive, where $0$ (resp.~$1$) means that the game is inactive (resp.~active).
\end{itemize}
The formal definition is as follows.

\begin{definition}
For $n \in \uint$, we define
\begin{align}
\short_n
&\defeq{=} \begin{cases}
\{\es\} \times \bin &\,\,\text{if}\,\, n = 0,\\
2^{\short_{n-1}} \times \bin &\,\,\text{if}\,\, n \geq 1.
\end{cases}
\end{align}
A \emph{game with activeness} is an element of
\begin{align}
\short &\defeq{=} \bigcup_{n \in \uint} \short_n.
\end{align}
\end{definition}

We write $G^g$ for $(G, g)$.
For $G^g, H^h \in \short$, we write $H^h \in G^g$ (resp.~$H^h \subseteq G^g$) for the condition $H^h \in G$ (resp.~$H \subseteq G$) by abusing notation.
If $G^g, H^h \in \short$ are the same elements of $\short$, then we write $G^g \cong H^h$.
We have $\short_0 \subsetneq \short_1 \subsetneq \short_2 \subsetneq \cdots$ by the definition.

Hereinafter, we refer to a game with activeness as ``game'' simply, unless otherwise specified.

For $G^g \in \short$, a game reachable by some number, possibly zero, of moves from $G^g$ is called a subposition of $G^g$ as follows.
\begin{definition}
\label{def:subpos}
A game $H^h$ is called a \emph{subposition} of a game $G^g$ if
there exist $n \in \uint$ and a sequence $(G^{g_0}_0, G^{g_1}_1, \ldots, G^{g_n}_n)$ of games such that
$G^{g_0}_0 \cong G^g, G^{g_n}_n \cong H^{h}$, and $G^{g_0}_0 \owns G^{g_1}_1 \owns \cdots \owns G^{g_n}_n$.
\end{definition}
Note that a game $G^g$ is a subposition of $G^g$ itself.

\begin{example}
\begin{align}
\short_0 &= \{\es\} \times \bin = \{\es^0, \es^1\},\\
\short_1 &= 2^{\short_0} \times \bin
= \{\es, \{\es^0\}, \{\es^1\},  \{\es^0, \es^1\}\} \times \bin
= \{\es^0, \es^1, \{\es^0\}^0, \{\es^0\}^1, \{\es^1\}^0, \{\es^1\}^1, \{\es^0, \es^1\}^0, \{\es^0, \es^1\}^1\},\\
\short_2 &= 2^{\short_1} \times \bin \owns G^g \defeq\cong \{\es^1, \{\es^1\}^0, \{\es^1\}^1, \{\es^0, \es^1\}^0\}^1.
\end{align}
The subpositions of $G^g$ are $\es^0, \es^1, \{\es^1\}^0, \{\es^1\}^1, \{\es^0, \es^1\}^0$, and $G^g$.
\end{example}

Many proofs in this paper are carried out by induction on $\birth(G^g$), defined as follows.
\begin{definition}
For any $G^g \in \short$, we define a mapping $\birth \colon \short \to \uint$ as
\begin{align}
\birth(G^g) = 
\begin{cases}
0 &\,\,\text{if}\,\, G = \es,\\
1 + \max_{G'^{g'} \in G^g} \birth(G'^{g'}) &\,\,\text{if}\,\, G \neq \es
\end{cases}
\end{align}
for $G^g \in \short$.
\end{definition}
Note that for any $G^g, G'^{g'} \in \short$, if $G'^{g'} \in G^g$, then $\birth(G'^{g'}) < \birth(G^g)$ by the definition.

Also, the following definition is used to state theorems in this paper.
\begin{definition}
We define a mapping $\type \colon \short \to \{0, 1_{\exists 0}, 1_{\forall 1}\}$ as
\begin{align}
\type(G^g) =
\begin{cases}
0 &\,\,\text{if}\,\,g = 0,\\
1_{\exists 0} &\,\,\text{if}\,\,g = 1 \,\,\text{and}\,\, \exists G'^{g'} \in G^g \,\,\text{s.t.}\,\, g' = 0,\\
1_{\forall 1} &\,\,\text{if}\,\,g = 1 \,\,\text{and}\,\, \forall G'^{g'} \in G^g, g' = 1
\end{cases}
\label{eq:qftmhuw4xm1w}
\end{align}
for $G^g \in \short$.
\end{definition}

\subsubsection{The Outcome Classes}

We define the outcome classes for games with activeness modifying Definition \ref{def:imp-outcome}
so that an inactive game immediately ends and thus is in $\PP$ by the normal-play convention.

\begin{definition}
We define a mapping $o\colon \short \to \{\PP, \NN\}$ as
\begin{align}
o(G^g) \defeq{=}
\begin{cases}
\NN &\,\,\text{if}\,\, g = 1 \,\,\text{and}\,\, \exists G'^{g'} \in G^g \,\,\text{s.t.}\,\, o(G'^{g'}) = \PP,\\
\PP &\,\,\text{if}\,\, g = 0 \,\,\text{or}\,\, \forall G'^{g'} \in G^g, o(G'^{g'}) = \NN
\end{cases}
\label{eq:outcome}
\end{align}
for $G^g \in \short$.
By abusing notation, we define $\PP \defeq{=} \{G^g \in \short : o(G^g) = \PP\}$ and $\NN \defeq{=} \{G^g \in \short : o(G^g) = \NN\}$.
\end{definition}

Note that for any $G^g \in \short$, the following equivalences hold:
\begin{align}
G^g \in \NN &\iff g = 1 \,\,\text{and}\,\, G \not\subseteq \NN,
\end{align}
equivalently,
\begin{align}
G^g \in \PP &\iff g = 0 \,\,\text{or}\,\, G \subseteq \NN.  \label{eq:outcome2}
\end{align}

For simplicity, we write $G^{g_0g_1g_2\ldots g_n}$ for $\{ \cdots \{\{\{G^{g_0}\}^{g_1}\}^{g_2}\} \cdots \}^{g_n}$,
where $G^{g_0} \in \short$, $n \in \uint$, and $g_1, g_2, \ldots, g_n \in \bin$.

\begin{example}~
\label{ex:outcome}
\begin{enumerate}[(i)]
\item $o(\es^0) = o(\es^1) = \PP$ because $\es^0$ and $\es^1$ have no options (in particular, every option is in $\NN$).
\item $o(\{\es^0, \es^1\}^0) = \PP$ because $\{\es^0, \es^1\}^0$ is inactive.
\item $o(\es^{11}) = \NN$ because $\es^{11}$ is active and has an option $\es^{1} \in \PP$ by (i).
\item $o(\es^{111}) = \PP$ because the only option of $\es^{111}$ is $\es^{11} \in \NN$ by (iii).
\end{enumerate}
\end{example}

\subsubsection{Addition of Games}

The sum $G^g + H^h$ is defined to be active if and only if at least one of $G^g$ or $H^h$ is active.
We formally define the sum of games with activeness modifying Definition \ref{def:imp-sum} as follows.

\begin{definition}
For $G^g, H^h \in \short$, we define
\begin{align}
G^g + H^h \defeq\cong \left( \{G'^{g'}+H^h : G'^{g'} \in G^g\} \cup \{G^g+H'^{h'} : H'^{h'} \in H^h\} \right)^{g \lor h}, \label{eq:sum}
\end{align}
where $g \lor h \defeq= \max\{g, h\}$.
\end{definition}

\begin{example}
Consider the sum of $\{\es^0, \es^1\}^0$ and $\es^{001}$.
By the definition, we have
\begin{align}
\{\es^0, \es^1\}^0 + \es^{001} \cong \{ \es^0 + \es^{001}, \es^1 + \es^{001}, \{\es^0, \es^1\}^0 + \es^{00} \}^1. \label{eq:hgom2w4urs14}
\end{align}
Intuitively, $\{\es^0, \es^1\}^0 + \es^{001}$ is the active game in which a player makes a move on exactly one of $\{\es^0, \es^1\}^0$ and $\es^{001}$ in a turn.
Namely, the player has the following three options:
\begin{itemize}
\item the move to $\es^0 + \es^{001}$ playing on the first component $\{\es^0, \es^1\}^0$,
\item the move to $\es^1 + \es^{001}$ playing on the first component $\{\es^0, \es^1\}^0$,
\item the move to $ \{\es^0, \es^1\}^0 + \es^{00}$ playing on the second component $\es^{001}$.
\end{itemize}

Note that since $\{\es^0, \es^1\}^0$ is inactive, it ends immediately and is in $\PP$ when played alone;
however, when played as the sum with the active game $\es^{001}$, the sum does not end immediately and has the three options mentioned above, including the moves playing on the inactive component $\{\es^0, \es^1\}^0$.

The right-hand side of \eqref{eq:hgom2w4urs14} is further expanded to $\{ \es^{001}, \es^{111}, \{ \es^{00}, \es^{11}, \{\es^0, \es^1\}^0 \}^0 \}^1$
as seen later in Example \ref{ex:sum2}.
\end{example}

When restricted to games $G^g$ such that all subpositions of $G^g$ are active,
it never occurs that the game ends because all terms become inactive.
Hence, their behavior in terms of the outcome and sums is identical to that of games without activeness.
In this sense, all games without activeness are included in the set of games with activeness.
More formally, $G \in \imp$ can be identified as $f(G) \in \short$ by the mapping $f \colon \imp \to \short$ defined as
 $f(G) \defeq= \{f(G'^{g'}) : G' \in G\}^1$ for $G \in \imp$.
 Namely, our framework of games with activeness is a generalization of games without activeness.

Then, to what extent do games with activeness inherit properties of games without activeness?
First, the addition of games satisfies the following basic properties, and thus $(\short, +)$ is a commutative monoid with an identity $\es^0$.

\begin{theorem}~
\label{thm:monoid}
\begin{enumerate}[(i)]
\item For any $G^g, H^h \in \short$, we have $G^g+H^h \cong H^h+G^g$.
\item For any $G^g \in \short$, we have $G^g+\es^0 \cong G^g$.
\item For any $G^g, H^h, J^j \in \short$, we have $(G^g+H^h)+J^j \cong G^g+(H^h+J^j)$.
\end{enumerate}
\end{theorem}
\begin{proof}[Proof of Theorem \ref{thm:monoid}]
(Proof of (i))
We prove this by induction on $\birth(G^g) + \birth(H^h)$.  We have
\begin{align*}
G^g+H^h
&\eqlab{A}\cong \left( \{G'^{g'}+H^h : G'^{g'} \in G^g\} \cup \{G^g+H'^{h'} : H'^{h'} \in H^h\} \right)^{g \lor h}\\
&\eqlab{B}\cong \left( \{H^h + G'^{g'} : G'^{g'} \in G^g\} \cup \{H'^{h'} + G^g : H'^{h'} \in H^h\} \right)^{g \lor h}\\
&\cong \left( \{H'^{h'} + G^g : H'^{h'} \in H^h\} \cup \{H^h + G'^{g'} : G'^{g'} \in G^g\} \right)^{h \lor g}\\
&\eqlab{C}\cong H^h + G^g,
\end{align*}
where
(A) follows from the definition of $G^g + H^h$,
(B) follows from the induction hypothesis,
and (C) follows from the definition of $H^h + G^g$.

(Proof of (ii))
We prove this by induction on $\birth(G^g)$. We have
\begin{align*}
G^g+\es^0
&\eqlab{A}\cong \left( \{G'^{g'}+\es^0 : G'^{g'} \in G^g\} \cup \{G^g+H'^{h'} : H'^{h'} \in \es^0\} \right)^{g \lor 0}\\
&\cong \{G'^{g'}+\es^0 : G'^{g'} \in G^g\}^g\\
&\eqlab{B}\cong \{G'^{g'}: G'^{g'} \in G^g\}^g\\
&\cong G^g,
\end{align*}
where
(A) follows from the definition of $G^g + \es^0$,
and (B) follows from the induction hypothesis.

(Proof of (iii))
We prove this by induction on $\birth(G^g) + \birth(H^h) +\birth(J^j)$. 
We have
\begin{align}
(G^g+H^h)+J^j
&\eqlab{A}\cong \left( \left\{K^k+J^j : K^k \in (G^g+H^h) \right\}
\cup \left\{(G^g+H^h)+J'^{j'} : J'^{j'} \in J^j \right\} \right)^{(g \lor h) \lor j}\\
&\eqlab{B}\cong \Big( \left\{K^k+J^j : K^k \in \{G'^{g'}+H^h : G'^{g'} \in G^g\} \cup \{G^g+H'^{h'} : H'^{h'} \in H^h\}\right\} \\
&\qquad \cup \left\{(G^g+H^h)+J'^{j'} : J'^{j'} \in J^j \right\} \Big)^{g \lor h \lor j}\\
&\cong \Big( \left\{(G'^{g'}+H^h)+J^j : G'^{g'} \in G^g \right \} \cup \left \{(G^g+H'^{h'})+J^j : H'^{h'} \in H^h \right\}\\
&\qquad \cup \left\{(G^g+H^h)+J'^{j'} : J'^{j'} \in J^j \right\} \Big)^{g \lor h \lor j}, \label{eq:0ekni0lmollj}
\end{align}
where
(A) follows from the definition of $(G^g+H^h)+J^j$,
and (B) follows from the definition of $G^g+H^h$.
Similarly, we have
\begin{align}
G^g+(H^h+J^j)
&\cong \Big( \left\{G'^{g'}+(H^h+J^j) : G'^{g'} \in G^g \right\}
\cup \left\{G^g+(H'^{h'}+J^j) : H'^{h'} \in H^h \right\} \\
&\qquad \cup \left\{G^g+(H^h+J'^{j'}) : J'^{j'} \in J^j \right\} \Big)^{g \lor h \lor j}. \label{eq:uatrberlgux4}
\end{align}
The right-hand sides of \eqref{eq:0ekni0lmollj} and \eqref{eq:uatrberlgux4} are identical by the induction hypothesis.
\end{proof}

By Theorem \ref{thm:monoid} (iii), we can write $(G^g+H^h)+J^j$ (equivalently, $G^g+(H^h+J^j)$) as $G^g+H^h+J^j$ without ambiguity.

\begin{example}
\label{ex:sum2}
We have
\begin{align}
\es^1 + \es^{00}
\eqlab{A}\cong \{ \es^1 + \es^{0} \}^1
\eqlab{B}\cong \{ \es^1 \}^1
\cong \es^{11}, \label{eq:jq7gu4ead998}
\end{align}
where
(A) follows from the definition of the sum,
and (B) follows from Theorem \ref{thm:monoid} (ii).
Hence, we have
\begin{align} 
\es^1 + \es^{001}
\eqlab{A}\cong \{\es^1 + \es^{00}\}^1
\eqlab{B}\cong \{\es^{11}\}^1
\cong \es^{111}, \label{eq:wv9yg6iv6gi6}
\end{align}
where
(A) follows from the definition of the sum,
and (B) follows from \eqref{eq:jq7gu4ead998}.
Also, we have
\begin{align}
\{\es^0, \es^1\}^0 + \es^{00}
&\eqlab{A}\cong \{ \es^0+ \es^{00}, \es^1 + \es^{00}, \{\es^0, \es^1\}^0 + \es^{0} \}^0\\
&\eqlab{B}\cong \{ \es^{00}, \es^1 + \es^{00}, \{\es^0, \es^1\}^0 \}^0\\
&\eqlab{C}\cong \{ \es^{00}, \es^{11}, \{\es^0, \es^1\}^0 \}^0, \label{eq:tzviyaselwkv}
\end{align}
where
(A) follows from the definition of the sum,
(B) follows from Theorem \ref{thm:monoid} (ii),
and (C) follows from \eqref{eq:jq7gu4ead998}.
Therefore, 
\begin{align}
\{\es^0, \es^1\}^0 + \es^{001}
&\eqlab{A}\cong \{ \es^0 + \es^{001}, \es^1 + \es^{001}, \{\es^0, \es^1\}^0 + \es^{00} \}^1\\
&\eqlab{B}\cong \{ \es^{001}, \es^1 + \es^{001}, \{\es^0, \es^1\}^0 + \es^{00} \}^1\\
&\eqlab{C}\cong \{ \es^{001}, \es^{111}, \{ \es^{00}, \es^{11}, \{\es^0, \es^1\}^0 \}^0 \}^1,
\end{align}
where
(A) follows from the definition of the sum,
(B) follows from Theorem \ref{thm:monoid} (ii),
and (C) follows from \eqref{eq:wv9yg6iv6gi6} and \eqref{eq:tzviyaselwkv}.
\end{example}

In the case of games without activeness, the following statements (i)--(iii) hold.
\begin{enumerate}[(i)]
\item For any $G \in \imp$, we have $G + G \in \PP$.
\item For any $G, H \in \imp$, if $G \in \PP$ and $H \in \PP$, then $G + H \in \PP$.
\item For any $G, H, J \in \imp$, if $G + H \in \PP$ and $H + J\in \PP$, then $G + J \in \PP$.
\end{enumerate}
We consider whether the following analogous statements (i')--(iii') hold for games with activeness.
\begin{enumerate}[(i')]
\item For any $G^g \in \short$, it holds that $G^g + G^g \in \PP$.
\item For any $G^g, H^h \in \short$, if $G^g \in \PP$ and $H^h \in \PP$, then $G^g + H^h \in \PP$.
\item For any $G^g, H^h, J^j \in \short$, if $G^g + H^h \in \PP$ and $H^h + J^j\in \PP$, then $G^g + J^j \in \PP$.
\end{enumerate}
The statements (ii') and (iii') are false because of the following counterexamples:
$G^g \defeq\cong \es^{00}, H^h \defeq\cong \es^1$ for (ii');
$G^g \defeq\cong \es^1, H^h \defeq\cong \es^0, J^j \defeq\cong \es^{00}$ for (iii').
On the other hand, the statement (i') holds true in a slightly stronger form as follows.

\begin{theorem}
\label{prop:twice}
For any $G^g \in \short$ and $b \in \bin$, we have $o(G^g+G^g+\es^b) =\PP$.
\end{theorem}
\begin{proof}[Proof of Theorem \ref{prop:twice}]
We prove this by induction on $\birth(G^g)$.
We consider the following two cases separately: the case $g \lor b = 0$ and the case $g  \lor b = 1$.
\begin{itemize}
\item The case $g \lor b = 0$: We have $g \lor g \lor b = g \lor b = 0$, so that $o(G^g + G^g + \es^b) = \PP$.
\item The case $g  \lor b = 1$:
We have $g \lor g \lor b = g \lor b = 1$.
To prove $o(G^g+G^g+\es^b) =\PP$, it remains to show that every option of $G^g+G^g+\es^b$ is in $\NN$.
Without loss of generality, it suffices to consider only options in the form $G'^{g'}+G^g+\es^b$ for some $G'^{g'} \in G^g$.
Then we have $o(G'^{g'} + G^g + \es^b) = \NN$ because
$g' \lor g \lor b \geq g \lor b = 1$ and
\begin{align}
G'^{g'} + G^g + \es^b \owns G'^{g'} + G'^{g'} + \es^b \eqlab{A}\in \PP,
\end{align}
where (A) follows from the induction hypothesis.
\end{itemize}
\end{proof}

\subsubsection{Nimbers and Grundy Numbers with Activeness}

In this section, we consider games with activeness corresponding to nimbers.
Unlike the case of games without activeness, there is room to freely define the activeness of the $n$-th nimber for each $n \in \uint$.
Therefore, we introduce a sequence $\bm{\gamma} \in \bin^{\infty}$ that specifies the activeness for each $n \in \uint$,
and we define nimbers for each sequence $\bm{\gamma}$ as the next definition,
where $\bin^{\infty}$ denotes the set of all infinite sequences over $\bin$.

\begin{definition}
\label{def:nimber}
Let $\bm{\gamma} = (\gamma_0, \gamma_1, \gamma_2, \ldots, ) \in \bin^{\infty}$.
We define $\st^{\bm{\gamma}} i \defeq{\cong} \{\st^{\bm{\gamma}} j : 0 \leq j < i\}^{\gamma_i}$ for $i \in \uint$ recursively.
\end{definition}

Note that even if $\bm{\gamma} \defeq = (\gamma_0, \gamma_1, \gamma_2, \ldots)$ and
$\bm{\gamma}' \defeq = (\gamma'_0, \gamma'_1, \gamma'_2, \ldots)$ are not equal,
if the first $(n+1)$ elements of them are identical (i.e., $(\gamma_0, \gamma_1, \ldots, \gamma_n) = (\gamma'_0, \gamma'_1, \ldots, \gamma'_n)$),
then we have $\st^{\bm{\gamma}} n \cong \st^{\bm{\gamma}'} n$.

We next consider a concept corresponding to Grundy numbers.
The Grundy number $\grundy(G)$ of a game without activeness $G$ is the unique integer $n \in \uint$ such that $o(G + \st n) = \PP$.
Based on that, we generalize Grundy numbers to games with activeness.
However, it is not necessarily the case that there exists exactly one $n \in \uint$ with $o(G^g + \st^{\bm{\gamma}} n) = \PP$,
and thus we define the object corresponding to Grundy numbers as the set of all $n \in \uint$ with $o(G^g + \st^{\bm{\gamma}} n) = \PP$ as follows.

\begin{definition}
\label{def:grundy-g}
For $\bm{\gamma} \in \bin^{\infty}$ and $G^g \in \short$,
we define a mapping $\grundy^{\bm{\gamma}} \colon \short \to 2^{\uint}$ as
\begin{align}
\grundy^{\bm{\gamma}}(G^g) \defeq= \{n \in \uint : o(G^g + \st^{\bm{\gamma}} n) = \PP\}.
\end{align}
\end{definition}

Then we have the following recursive formula on $\grundy^{\bm{\gamma}}$ corresponding to \eqref{eq:ofe48aki7qcy}.

\begin{theorem}
\label{thm:grundy}
For any $\bm{\gamma} = (\gamma_0, \gamma_1, \gamma_2, \ldots, ) \in \bin^{\infty}$ and $G^g \in \short$, we have
\begin{align}
\grundy^{\bm{\gamma}}(G^g) = \{n \in \uint : g \lor \gamma_n = 0\} \cup 
\left\{ n \in \uint : \{0, 1, 2, \ldots, n-1\} \subseteq U, n \not\in U \right\}, \label{eq:awmlfqha8vf1}
\end{align}
where
\begin{align}
\label{eq:yi5wdg1pvkgv}
U \defeq= \bigcup_{G'^{g'} \in G^g} \grundy^{\bm{\gamma}}(G'^{g'}).
\end{align}
\end{theorem}

Note that $\tilde{U} \defeq=\left\{ n \in \uint : \{0, 1, 2, \ldots, n-1\} \subseteq U, n \not\in U \right\}$ in \eqref{eq:awmlfqha8vf1}
has at most one element, and if $\tilde{U}$ has an element, then the unique element is equal to $\mex U$.

\begin{proof}[Proof of Theorem \ref{thm:grundy}]
Define
\begin{align}
\tilde{U} \defeq= \left\{ n' \in \uint : \{0, 1, 2, \ldots, n'-1\} \subseteq U, n' \not\in U \right\}. \label{eq:3y2p544cdlxt}
\end{align}
To prove the theorem, it suffices to show that the following statements (i) and (ii) hold.
\begin{enumerate}[(i)]
\item For any $n \in \uint$, if $g \lor \gamma_n = 0$, then $n \in \grundy^{\bm{\gamma}}(G^g)$.
\item For any $n \in \uint$, if $g \lor \gamma_n = 1$, then the following equivalence holds: $n \in \tilde{U} \iff n \in \grundy^{\bm{\gamma}}(G^g)$.
\end{enumerate}
The statement (i) holds since the assumption $g \lor \gamma_n = 0$ directly implies $o(G^g + \st^{\bm{\gamma}} n) = \PP$.
We show (ii) by induction on $n \in \uint$.
We consider the following two cases separately: the case $n \in \tilde{U}$ and the case $n \not\in \tilde{U}$.

\begin{itemize}
\item The case $n \in \tilde{U}$:
We have $o(G^g + \st^{\bm{\gamma}} n) = \PP$ so that $n \in \grundy^{\bm{\gamma}}(G^g)$ because
every option of $G^g + \st^{\bm{\gamma}} n$ is in $\NN$ as follows.
\begin{itemize}
\item For any $G'^{g'} \in G^g$, the option $G'^{g'} + \st^{\bm{\gamma}} n$ is in $\NN$ because
$n \in \tilde{U}$ implies $n \not\in U$ by \eqref{eq:3y2p544cdlxt},
which leads to $n \not\in \grundy^{\bm{\gamma}}(G'^{g'})$ by \eqref{eq:yi5wdg1pvkgv}.
\item For any $n' \in \{0, 1, 2, \ldots, n-1\}$, the option $G^{g} + \st^{\bm{\gamma}} n'$ is in $\NN$
because the assumption $n \in \tilde{U}$ and \eqref{eq:3y2p544cdlxt} lead to $n' \in  \{0, 1, 2, \ldots, n-1\} \subseteq U$ and thus $n' \not\in \tilde{U}$, which is equivalent to $n' \not\in \grundy^{\bm{\gamma}}(G^g)$ by the induction hypothesis.
Note that $n' \not\in \grundy^{\bm{\gamma}}(G^g)$ is equivalent to $o(G^{g} + \st^{\bm{\gamma}} n') = \NN$ as desired.
\end{itemize}

\item The case $n \not\in \tilde{U}$:
By \eqref{eq:3y2p544cdlxt}, the following two cases are possible: the case where $n' \not\in U$ for some $n' \in \{0, 1, 2, \ldots, n-1\}$,
and the case $n \in U$.
To prove $n \not\in \grundy^{\bm{\gamma}}(G^g)$ (equivalently, $o(G^g + \st^{\bm{\gamma}} n) = \NN$),
we show that $G^g + \st^{\bm{\gamma}} n$ has an option in $\PP$ for the two cases as follows.
\begin{itemize}
\item The case where $n' \not\in U$ for some $n' \in \{0, 1, 2, \ldots, n-1\}$:
Then $m \defeq= \mex U$ is well-defined, and $m \leq n' < n$ holds.
Since $m \in \tilde{U}$ by \eqref{eq:3y2p544cdlxt}, we have $m \in \grundy^{\bm{\gamma}}(G^g)$ by induction hypothesis.
Therefore, the option $G^g + \st^{\bm{\gamma}} m$ is in $\PP$.

\item The case $n \in U$:
Then there exists $G'^{g'} \in G^g$ such that $n \in \grundy^{\bm{\gamma}}(G'^{g'})$ by \eqref{eq:yi5wdg1pvkgv},
so that the option $G'^{g'} + \st^{\bm{\gamma}} n$ is in $\PP$.
\end{itemize}
\end{itemize}
\end{proof}

Let $\bm{0} \defeq= (0, 0, 0, \ldots ) \in \bin^{\infty}$ and $\bm{1} \defeq= (1, 1, 1, \ldots ) \in \bin^{\infty}$.
In the particular case $\bm{\gamma} = \bm{0}$ or $\bm{\gamma} = \bm{1}$,
Theorem \ref{thm:grundy} deduces the following formulas.

\begin{theorem}
\label{thm:grundy-01}
For any $G^g \in \short$, the following statements (i) and (ii) hold.
\begin{enumerate}[(i)]
\item $\grundy^{\bm{1}}(G^g) = \{\mex \bigcup_{G'^{g'} \in G^g} \grundy^{\bm{1}}(G'^{g'})\}$.
\item \begin{align}
\grundy^{\bm{0}}(G^g)
= \begin{cases}
\uint &\,\,\text{if}\,\, \type(G^g) = 0,\\
\es &\,\,\text{if}\,\, \type(G^g) = 1_{\exists 0},\\
\left\{\mex \displaystyle\bigcup_{\substack{G'^{g'} \in G^g,\\ \type(G'^{g'}) = 1_{\forall 1}}} \grundy^{\bm{0}}(G'^{g'}) \right\} &\,\,\text{if}\,\, \type(G^g) = 1_{\forall 1}.\\
\end{cases}
\end{align}
\end{enumerate}
\end{theorem}

\begin{proof}[Proof of Theorem \ref{thm:grundy-01}]
(Proof of (i))
We prove this by induction on $\birth(G^g)$.
For any $G'^{g'} \in G^g$, we have $|\grundy^{\bm{1}}(G'^{g'})| = |\{\mex \bigcup_{G''^{g''} \in G'^{g'}} \grundy^{\bm{1}}(G''^{g''})\}| = 1$ by the induction hypothesis.
Hence, 
\begin{align}
U \defeq= \bigcup_{G'^{g'} \in G^g} \grundy^{\bm{1}}(G'^{g'}), \label{eq:1qlaqu5xdios}
\end{align}
is a finite set, so that $\mex U$ is well-defined, and
\begin{align}
\left\{ n \in \uint : \{0, 1, 2, \ldots, n-1\} \subseteq U, n \not\in U \right\} = \{\mex U\}. \label{eq:z9hc2lbkw6k9}
\end{align}
Therefore,
\begin{align}
\grundy^{\bm{1}}(G^g)
&\eqlab{A}=  \{n \in \uint : g \lor 1 = 0\} \cup \left\{ n \in \uint : \{0, 1, 2, \ldots, n-1\} \subseteq U, n \not\in U \right\}\\
&\eqlab{B}=  \es \cup \left\{ n \in \uint : \{0, 1, 2, \ldots, n-1\} \subseteq U, n \not\in U \right\}\\
&\eqlab{C}=  \{ \mex U\}\\
&\eqlab{D}= \left\{\mex \bigcup_{G'^{g'} \in G^g} \grundy^{\bm{1}}(G'^{g'})\right\},
\end{align}
where
(A) follows from Theorem \ref{thm:grundy},
(B) follows from $g \lor 1 = 1 \neq 0$,
(C) follows from \eqref{eq:z9hc2lbkw6k9},
and (D) follows from \eqref{eq:1qlaqu5xdios}.

(Proof of (ii))
Define
\begin{align}
U \defeq= \bigcup_{G'^{g'} \in G^g} \grundy^{\bm{0}}(G'^{g'}). \label{eq:4rx21mh99vdu}
\end{align}
We consider the following three cases separately:
the case $\type(G^g) = 0$, the case $\type(G^g) = 1_{\exists 0}$, and the case $\type(G^g) = 1_{\forall 1}$.
\begin{itemize}
\item The case $\type(G^g) = 0$: We have
\begin{align}
\grundy^{\bm{0}}(G^g)
\eqlab{A}\supseteq \{n \in \uint : g \lor 0 = 0\}
\eqlab{B}= \{n \in \uint : 0 \lor 0 = 0\}
= \uint,
\end{align}
where
(A) follows from Theorem \ref{thm:grundy},
and (B) follows from $\type(G^g) = 0$.

\item The case $\type(G^g) = 1_{\exists 0}$:
Then there exists $G'^{g'} \in G^g$ such that $g' = 0$, so that
\begin{align}
U \eqlab{A}= \bigcup_{\hat{G}^{\hat{g}} \in G^g} \grundy^{\bm{0}}(\hat{G}^{\hat{g}})
\supseteq\grundy^{\bm{0}}(G'^{g'})
\eqlab{B}= \uint, \label{eq:ubxf5ih92648}
\end{align}
where
(A) follows from \eqref{eq:4rx21mh99vdu},
and (B) follows from the case $\type(G^g) = 0$ shown above since $\type(G'^{g'}) = 0$.
Hence, we obtain
\begin{align}
\grundy^{\bm{0}}(G^g)
&\eqlab{A}=  \{n \in \uint : g \lor 0 = 0\} \cup \left\{ n \in \uint : \{0, 1, 2, \ldots, n-1\} \subseteq U, n \not\in U \right\}\\
&\eqlab{B}=  \es \cup \left\{ n \in \uint : \{0, 1, 2, \ldots, n-1\} \subseteq U, n \not\in U \right\}\\
&\eqlab{C}=  \left\{ n \in \uint : \{0, 1, 2, \ldots, n-1\} \subseteq \uint, n \not\in \uint \right\}\\
&= \es,
\end{align}
where
(A) follows from Theorem \ref{thm:grundy},
(B) follows from $g \lor 0 = 1 \lor 0 = 1 \neq 0$,
and (C) follows from \eqref{eq:ubxf5ih92648}.

\item The case $\type(G^g) = 1_{\forall 1}$:
We prove this by induction on $\birth(G^g)$.
We have
\begin{align}
\hspace{-20pt}
U \eqlab{A}= \bigcup_{G'^{g'} \in G^g} \grundy^{\bm{0}}(G'^{g'})
\eqlab{B}= \left(\bigcup_{\substack{G'^{g'} \in G^g,\\ \type(G'^{g'}) = 1_{\exists 0}}} \grundy^{\bm{0}}(G'^{g'})\right) \cup \left(\bigcup_{\substack{G'^{g'} \in G^g,\\ \type(G'^{g'}) = 1_{\forall 1}}} \grundy^{\bm{0}}(G'^{g'})\right)
\eqlab{C}= \bigcup_{\substack{G'^{g'} \in G^g,\\ \type(G'^{g'}) = 1_{\forall 1}}} \grundy^{\bm{0}}(G'^{g'}),
\label{eq:59uybh2r8r63}
\end{align}
where
(A) follows from \eqref{eq:4rx21mh99vdu},
(B) follows from $\type(G^g) = 1_{\forall 1}$,
and (C) follows from the case $\type(G^g) = 1_{\exists 0}$ shown above.
In the right-hand side of \eqref{eq:59uybh2r8r63}, we have $|\grundy^{\bm{0}}(G'^{g'})| = 1$ for every $G'^{g'}$ by the induction hypothesis.
Hence, $U$ is finite, so that $\mex U$ is well-defined, and
\begin{align}
\left\{ n \in \uint : \{0, 1, 2, \ldots, n-1\} \subseteq U, n \not\in U \right\} = \{\mex U\}. \label{eq:ijviog0x2jp2}
\end{align}
Therefore,
\begin{align}
\grundy^{\bm{0}}(G^g)
&\eqlab{A}=  \{n \in \uint : g \lor 0 = 0\} \cup \left\{ n \in \uint : \{0, 1, 2, \ldots, n-1\} \subseteq U, n \not\in U \right\}\\
&\eqlab{B}=  \es \cup \left\{ n \in \uint : \{0, 1, 2, \ldots, n-1\} \subseteq U, n \not\in U \right\}\\
&\eqlab{C}=  \{ \mex U\}\\
&\eqlab{D}= \left\{\mex \displaystyle\bigcup_{\substack{G'^{g'} \in G^g,\\ \type(G'^{g'}) = 1_{\forall 1}}} \grundy^{\bm{0}}(G'^{g'}) \right\},
\end{align}
where
(A) follows from Theorem \ref{thm:grundy},
(B) follows from $g \lor 0 = 1 \lor 0 = 1 \neq 0$,
(C) follows from \eqref{eq:ijviog0x2jp2},
and (D) follows from \eqref{eq:59uybh2r8r63}.
\end{itemize}
\end{proof}

As shown in Theorem \ref{thm:grundy-01} (i), the set $\grundy^{\bm{1}}(G^g)$ has exactly one element for any $G^g \in \short$,
and the unique element of $\grundy^{\bm{1}}(G^g)$ is the Grundy number of $G^g \in \short$ calculated ignoring activeness,
given as the following definition.

\begin{definition}
We define a mapping $\grundy \colon \short \to \uint$ as $\grundy(G^g) \defeq= \mex\{\grundy(G'^{g'}) : G'^{g'} \in G\}$
for $G^g \in \short$.
\end{definition}

\begin{corollary}
\label{cor:grundy}
For any $G^g \in \short$, we have $\grundy^{\bm{1}}(G^g) = \{\grundy(G^g)\}.$
\end{corollary}

Also, Theorem \ref{thm:grundy-01} (ii) shows that
for any $G^g \in \short$, if $\type(G^g) = 1_{\forall 1}$, then the set $\grundy^{\bm{0}}(G^g)$ has exactly one element, which is interpreted as the Grundy number of $G^g$ calculated ignoring activeness and allowing only moves to the games $H^h$
with $\type(H^h) = 1_{\forall 1}$.

\subsubsection{Equivalence among Games}

We introduce an equivalence relation $=$ of games with activeness in the same way as games without activeness (Definition \ref{def:imp-equiv}).

\begin{definition}
We define a binary relation $=$ of $\short$ as
\begin{align}
\{(G^g, H^h) \in \short \times \short : \forall X^x \in \short, o(G^g+X^x) = o(H^h+X^x)\}.
\end{align}
Namely, 
\begin{align}
G^g = H^h \defeq{\iff} \forall X^x \in \short, o(G^g+X^x) = o(H^h+X^x).
\end{align}
\end{definition}
It is easily seen that the binary relation $=$ is an equivalence relation of $\short$.

By the definition, $G^g \neq H^h$ if and only if
there exists $X^x \in \short$ such that $o(G^g + X^x) \neq o(H^h + X^x)$.

\begin{definition}
For any $G^g, H^h, X^x \in \short$, if $o(G^g + X^x) \neq o(H^h + X^x)$, then
$X^x$ is said to \emph{distinguish} $G^g$ from $H^h$.
\end{definition}
Namely, $G^g \neq H^h$ if and only if there exists $X^x \in \short$ that distinguishes $G^g$ from $H^h$.

We confirm basic properties of the equivalence relation $=$ as Theorems \ref{thm:equiv}--\ref{thm:option-neq}.

\begin{theorem}
\label{thm:equiv}
Let $G^g, H^h \in \short$ be arbitrary.
If $G^g = H^h$, then for any $\bm{\gamma} \in \bin^{\infty}$, we have $\grundy^{\bm{\gamma}}(G^g) = \grundy^{\bm{\gamma}}(H^h)$.
\end{theorem}

\begin{proof}[Proof of Theorem \ref{thm:equiv}]
Assume $G^g = H^h$ and $\grundy^{\bm{\gamma}}(G^g) \neq \grundy^{\bm{\gamma}}(H^h)$ for some $\bm{\gamma} \in \bin^{\infty}$.
Without loss of generality, we may assume that there exists $n \in \grundy^{\bm{\gamma}}(G^g) \setminus \grundy^{\bm{\gamma}}(H^h)$.
Then $o(G^g + \st^{\bm{\gamma}}n) = \PP$ and $o(H^h + \st^{\bm{\gamma}}n) = \NN$;
that is, $\st^{\bm{\gamma}}n$ distinguishes $G^g$ from $H^h$.
This contradicts $G^g = H^h$.
\end{proof}

By Theorem \ref{thm:equiv}, if $\grundy^{\bm{\gamma}}(G^g) \neq \grundy^{\bm{\gamma}}(H^h)$ for some $\bm{\gamma} \in \bin^{\infty}$, then it can be determined that $G^g \neq H^h$.
On the other hand, the converse implication does not hold.
Namely, there exist $G^g, H^h \in \short$ with $G^g \neq H^h$ such that any nimber $\st^{\bm{\gamma}} n$ does not distinguish them.
That is addressed later in Remark \ref{rem:transitive} as a stronger argument.

Also, Theorem \ref{thm:equiv} yields the following corollary.

\begin{corollary}
\label{cor:equiv}
For any $G^g, H^h \in \short$, if $G^g = H^h$, then the following statements (i)--(iii) hold.
\begin{enumerate}[(i)]
\item $o(G^g) = o(H^h)$.
\item $\type(G^g) = \type(H^h)$, in particular, $g = h$.
\item $\grundy(G^g) = \grundy(H^h)$.
\end{enumerate}
\end{corollary}

\begin{proof}[Proof of Corollary \ref{cor:equiv}]
(Proof of (i))
Applying Theorem \ref{thm:equiv} with $\bm{\gamma} = \bm{0}$, we have $\grundy^{\bm{0}}(G^g) = \grundy^{\bm{0}}(H^h)$;
in particular, $0 \in \grundy^{\bm{0}}(G^g)$ if and only if $0 \in \grundy^{\bm{0}}(H^h)$.
This yields $o(G^g) = o(G^g + \es^0) = o(G^g + \st^{\bm{0}} 0) = o(H^h + \st^{\bm{0}} 0) = o(H^h + \es^0) = o(H^h)$.

(Proof of (ii)) 
Applying Theorem \ref{thm:equiv} with $\bm{\gamma} = \bm{0}$, we have $\grundy^{\bm{0}}(G^g) = \grundy^{\bm{0}}(H^h)$.
This implies $\type(G^g) = \type(H^h)$ by Theorem \ref{thm:grundy-01} (ii).

(Proof of (iii)) Directly from Theorem \ref{thm:equiv} with $\bm{\gamma} = \bm{1}$ and Corollary \ref{cor:grundy}.
\end{proof}

\begin{theorem}
\label{thm:add}
For any $G^g, H^h, J^j, K^k\in \short$, if $G^g = H^h$ and $J^j = K^k$, then $G^g+J^j = H^h+K^k$.
\end{theorem}

\begin{proof}[Proof of Theorem \ref{thm:add}]
For any $A^a, B^b, C^c \in \short$, if $A^a = B^b$, then $A^a + C^c = B^b + C^c$ because
for any $X^x \in \short$, we have
\begin{align}
o((A^a+C^c)+X^x) = o(A^a+(C^c+X^x)) \eqlab{A}= o(B^b+(C^c+X^x)) = o((B^b+C^c)+X^x), \label{eq:pq4lkbv30lhp}
\end{align}
where (A) follows from $A^a = B^b$.
Applying this twice, we obtain
\begin{align}
G^g+J^j \eqlab{A}= H^h+J^j \eqlab{B}= H^h+K^k,
\end{align}
where
(A) follows from $G^g = H^h$ and \eqref{eq:pq4lkbv30lhp},
and (B) follows from $J^j = K^k$ and \eqref{eq:pq4lkbv30lhp}.
\end{proof}

\begin{theorem}
\label{thm:option-neq}
For any $G^g \in \short$ and $G'^{g'} \in G^g$, we have $G^g \neq G'^{g'}$.
\end{theorem}

\begin{proof}[Proof of Theorem \ref{thm:option-neq}]
A game $X^x \defeq\cong G^g + \es^1$ distinguishes $G^g$ from $G'^{g'}$:
\begin{align}
o(G^g + X^x) = o(G^g + G^g + \es^1) &\eqlab{A}= \PP,\\
o(G'^{g'} + X^x)= o(G'^{g'} + G^g + \es^1) &\eqlab{B}= \NN,
\end{align}
where
(A) follows from Theorem \ref{prop:twice},
and (B) follows from $G'^{g'} + G^g + \es^1 \owns G'^{g'} + G'^{g'} + \es^1 \in \PP$ by Theorem \ref{prop:twice}.
\end{proof}

Given two games $G^g$ and $H^h$, the following lemma gives a sufficient condition for the equivalence $G^g = H^h$.

\begin{lemma}
\label{lem:equal}
Let $G^g, H^g \in \short$ be arbitrary.
If for any $X^x \in \short$ with $g \lor x = 1$, the following conditions (a) and (b) hold, then $G^g = H^g$.
\begin{enumerate}[(a)]
\item If $o(G^g + X^x) = \PP$, then for any $H'^{h'} \in H^g$, it holds that $o(H'^{h'} + X^x) = \NN$.
\item If there exists $G'^{g'} \in G^g$ such that $o(G'^{g'} + X^x) = \PP$,
then there exists $H'^{h'} \in H^h$ such that $o(H'^{h'} + X^x) = \PP$.
\end{enumerate}
\end{lemma}

\begin{proof}[Proof of Lemma \ref{lem:equal}]
We show
\begin{align}
\forall X^x \in \short, o(G^g+X^x) = o(H^g+X^x)
\end{align}
by induction on $\birth(X^x)$.
Choose $X^x \in \short$ arbitrarily.
We consider the following two cases separately: the case $o(G^g+X^x) = \PP$ and the case $o(G^g+X^x) = \NN$.
\begin{itemize}
\item The case $o(G^g+X^x) = \PP$:
If $g \lor x = 0$, then $o(H^g+X^x) = \PP$ holds directly. 
If $g \lor x = 1$, then we confirm $o(H^g+X^x) = \PP$ because every option of $H^g+X^x$ is in $\NN$ as follows.
\begin{itemize}
\item For any $H'^{h'} \in H^g$, we have $o(H'^{h'} + X^x) = \NN$ directly from the assumption (a).
\item For any $X'^{x'} \in X^x$, we have
\begin{align}
o(H^g+X'^{x'}) \eqlab{A}= o(G^g + X'^{x'}) \eqlab{B}= \NN,
\end{align}
where
(A) follows from the induction hypothesis,
and (B) follows from $o(G^g+X^x) = \PP$ and $g \lor x = 1$.
\end{itemize}

\item The case $o(G^g+X^x) = \NN$:
Then at least one of the following two cases applies:
the case where $o(G'^{g'}+X^x) = \PP$ for some $G'^{g'} \in G^g$;
the case where $o(G^g+X'^{x'}) = \PP$ for some $X'^{x'} \in X^x$.
We confirm $o(H^g+X^x) = \NN$ because $g \lor x = 1$ by $o(G^g+X^x) = \NN$,
and the game $H^g+X^x$ has an option in $\PP$ for the two cases as follows.
\begin{itemize}
\item The case where $o(G'^{g'}+X^x) = \PP$ for some $G'^{g'} \in G^g$:
Then $H^g + X^x \owns H'^{h'}+X^{x} \in \PP$ for some $H'^{h'} \in \short$ by the assumption (b).
\item The case where $o(G^g+X'^{x'}) = \PP$ for some $X'^{x'} \in X^x$:
Then the option $H^g+X'^{x'}$ of $H^g + X^x$ satisfies
\begin{align}
o(H^g+X'^{x'}) \eqlab{A}= o(G^g + X'^{x'}) = \PP,
\end{align}
where (A) follows from the induction hypothesis.
\end{itemize}
\end{itemize}
\end{proof}

Using Lemma \ref{lem:equal}, we show the following theorem corresponding to Theorem \ref{cor:replace}.

\begin{theorem}[Replacement Lemma]
\label{prop:replace}
For any $\{G'^{g'_1}_1, G'^{g'_2}_2, \ldots, G'^{g'_n}_n\}^g \in \short$ and $H'^{h'_1}_1 \in \short$,
if $G'^{g'_1}_1 = H'^{h'_1}_1$, then
$\{G'^{g'_1}_1, G'^{g'_2}_2, \ldots, G'^{g'_n}_n\}^g =  \{H'^{h'_1}_1, G'^{g'_2}_2, \ldots, G'^{g'_n}_n \}^g$.
\end{theorem}

\begin{proof}[Proof of Theorem \ref{prop:replace}]
Define $G^g \defeq\cong \{G'^{g'_1}_1, G'^{g'_2}_2, \ldots, G'^{g'_n}_n\}^g$ and $H^g \defeq\cong \{H'^{h'_1}_1, G'^{g'_2}_2, \ldots, G'^{g'_n}_n\}^g$. 
It suffices to show that $G^g$ and $H^h$ satisfy the conditions (a) and (b) of Lemma \ref{lem:equal}.
To prove that, choose $X^x \in \short$ with
\begin{align}
g \lor x = 1 \label{eq:qm6qtioj81gj}
\end{align}
arbitrarily.

(Proof of Lemma \ref{lem:equal} (a))
Assume $o(G^g + X^x) = \PP$.
Then we confirm that every option $H'^{h'}$ of $H^g$ satisfies $o(H'^{h'} + X^x) = \NN$ as follows.
\begin{itemize}
\item For any $H'^{h'_1}_1\in H^g$, we have
\begin{align}
H'^{h'_1}_1+X^{x} \eqlab{A}= G'^{g'_1}_1 + X^{x} \in G^g + X^x \eqlab{B}\subseteq \NN,
\end{align}
where
(A) follows from $G'^{g'_1}_1 = H'^{h'_1}_1$,
and (B) follows from $o(G^g+X^x) = \PP$ and \eqref{eq:qm6qtioj81gj}.

\item For any $G'^{g'_i}_i \in H^g$ with $i \geq 2$, we have
$G'^{g'_i}_i + X^x \in G^g + X^x \subseteq \NN$ by $o(G^g+X^x) = \PP$ and \eqref{eq:qm6qtioj81gj}.
\end{itemize}

(Proof of Lemma \ref{lem:equal} (b))
Assume that $o(G'^{g'_i}_i + X^x) = \PP$ for some $G'^{g'_i}_i \in G^g$.
We consider the following two cases separately: the case $o(G'^{g'_1}_1+X^x) = \PP$ and the case where $o(G'^{g'_i}_i+X^x) = \PP$ for some $i \geq 2$.
\begin{itemize}
\item The case $o(G'^{g'_1}_1+X^x) = \PP$: Then $H^g + X^x \owns H'^{h'_1}_1+X^{x} = G'^{g'_1}_1 + X^{x} \in \PP$ by $G'^{g'_1}_1 = H'^{h'_1}_1$.
\item The case where $o(G'^{g'_i}_i+X^x) = \PP$ for some $i \geq 2$: Then $H^g + X^x \owns G'^{g'_i}_i + X^x \in \PP$.
\end{itemize}
\end{proof}

Suppose that a game $G^g$ has two distinct options $G'^{g'_1}_1$ and $G'^{g'_2}_2$ with $G'^{g'_1}_1 = G'^{g'_2}_2$.
The game $H^h$ obtained by replacing $G'^{g'_2}_2$ with $G'^{g'_1}_1$ (and removing the duplicate) is equivalent to $G^g$ by Theorem \ref{prop:replace}.
Therefore, we obtain the following corollary.

\begin{corollary}
For any $G^g \in \short$ and $G'^{g'_1}_1, G'^{g'_2}_2 \in G^g$, if $G'^{g'_1}_1 \not\cong G'^{g'_2}_2$ and $G'^{g'_1}_1 = G'^{g'_2}_2$ then $G^g = (G\setminus\{G'^{g'_1}_1\})^g$.
\end{corollary}

\subsection{Distinguishing Theorem}
\label{subsec:distinguish}

If $G^g \neq H^h$, then there exists a game $X^x$ that distinguishes $G^x$ from $H^x$.
Then it is not specified which of $o(G^g + X^x)$ and $o(H^h + X^x)$ is in $\PP$.
In this subsection, we show that if $G^g \neq H^h$, then there exists a game $X^1$ such that $o(G^g + X^1) = \PP$ and $o(H^h + X^1) = \NN$.
More generally, we consider a game $X^1$ such that
$o(G_1^{g_1} + X^1) =  o(G_2^{g_2} + X^1) = \cdots = o(G_m^{g_m} + X^1) =\PP$ and
$o(H_1^{h_1} + X^1) =  o(H_2^{h_2} + X^1) = \cdots = o(H_n^{h_n} + X^1) =\NN$,
and we give a necessary and sufficient condition for such a game $X^1$ exists as Theorem \ref{thm:distinguish}.
To state Theorem \ref{thm:distinguish}, we first introduce the following definition.

\begin{definition}
For $G^g, H^h \in \short$, if the following conditions (a) and (b) hold, then we write $G^g \bowtie H^h$
.
\begin{enumerate}[(a)]
\item For any $G'^{g'} \in G^g$, it holds that $G'^{g'} \neq H^h$.
\item For any $H'^{h'} \in H^h$, it holds that $G^{g} \neq H'^{h'}$.
\end{enumerate}
\end{definition}

Using the above definition, we state the desired theorem as follows.

\begin{theorem}[Distinguishing Theorem]
\label{thm:distinguish}
For any $m, n \in \uint$, and $G_1^{g_1}, G_2^{g_2}, \ldots, G_m^{g_m}, H_1^{h_1}, H_2^{h_2}, \ldots, H_n^{h_n}$,
the following conditions (a) and (b) are equivalent.
\begin{enumerate}[(a)]
\item There exists $X^1 \in \short$ such that
\begin{align}
o(G_i^{g_i} + X^1) &= \PP \,\,\text{for any}\,\,i \in [m],\label{eq:wtv7j5xqrzec}\\
o(H_j^{h_j} + X^1) &= \NN \,\,\text{for any}\,\,j \in [n]. \label{eq:zig7xesatpwr}
\end{align}
\item The following conditions (b1) and (b2) hold.
\begin{itemize}
\item[\rm{(b1)}] $G^{g_i}_i \bowtie G^{g_j}_j$ for any $i, j \in [m]$.
\item[\rm{(b2)}] $G^{g_i}_i \neq H^{h_j}_j$ for any $i \in [m], j \in [n]$.
\end{itemize}
\end{enumerate}
\end{theorem}

To prove Theorem \ref{thm:distinguish}, we first show results for particular cases as Lemmas \ref{lem:distinguish} and \ref{lem:distinguish2}.

\begin{lemma}
\label{lem:distinguish}
For any $G^g, H^h \in \short$, if $G^g \neq H^h$, then there exists $X^1 \in \short$ such that $o(G^g + X^1) = \PP$ and $o(H^h + X^1) = \NN$.
\end{lemma}

\begin{proof}[Proof of Lemma \ref{lem:distinguish}]
Assume $G^g \neq H^h$.
Then there exists $Y^y \in \short$ such that $o(G^g + Y^y) \neq o(H^h + Y^y)$.
We first consider the case 
\begin{align}
o(G^g + Y^y) &= \NN, \label{eq:t67swugdp8cq}\\
o(H^h + Y^y) &= \PP. \label{eq:oce6y97xyqnq}
\end{align}
Let $G^g \cong \{G'^{g'_1}_1, G'^{g'_2}_2, \ldots, G'^{g'_n}_n\}^g$,
and define $X^1 \defeq\cong \{Y^y, G'^{g'_1}_1+\es^1, G'^{g'_2}_2+\es^1, \ldots, G'^{g'_n}_n+\es^1\}^1$.

We show that $X^1$ satisfies the desired condition: $o(G^g + X^1) = \PP$ and $o(H^h + X^1) = \NN$.
We have $o(H^h + X^1) = \NN$ because $H^h + X^1 \owns H^h + Y^y \in \PP$ by \eqref{eq:oce6y97xyqnq}.

We confirm $o(G^g + X^1) = \PP$ because every option of $G^g + X^1$ is in $\NN$ as follows.
\begin{itemize}
\item For any $G'^{g'_i}_i \in G^g$, the option $G'^{g'_i}_i + X^1 \in G^g + X^1$ has an option $G'^{g'_i}_i + G'^{g'_i}_i + \es^1 \in \PP$ by Theorem \ref{prop:twice}, so that $G'^{g'_i}_i + X^1 \in \NN$.
\item For any $G'^{g'_i}_i \in G^g$, the option $G^g + (G'^{g'_i}_i + \es^1) \in G^g + X^1$ has an option $G'^{g'_i}_i + G'^{g'_i}_i + \es^1 \in \PP$ by Theorem \ref{prop:twice}, so that $G^g + (G'^{g'_i}_i + \es^1)  \in \NN$.
\item The option $G^g + Y^y \in G^g + X^1$ is in $\NN$ by \eqref{eq:t67swugdp8cq}.
\end{itemize}
In the other case than \eqref{eq:t67swugdp8cq} and \eqref{eq:oce6y97xyqnq} (i.e., the case $o(H^h + Y^y) = \NN$ and $o(G^g + Y^y) = \PP$),
we can obtain $X^1 \in \short$ with $o(H^h + X^1) = \PP$ and $o(G^g + X^1) = \NN$ by the discussion above,
and thus we can reduce the problem to the case \eqref{eq:t67swugdp8cq} and \eqref{eq:oce6y97xyqnq}.
\end{proof}

We next prove Lemma \ref{lem:distinguish2}, which is a strengthened version of Lemma \ref{lem:distinguish}.

\begin{lemma}
\label{lem:distinguish2}
For any $n \in \uint$ and $G^g, H_1^{h_1}, \ldots, H_n^{h_n} \in \short$, 
if $G^g \neq H_i^{h_i}$ for any $i \in [n]$, then 
there exists $X^1 \in \short$ such that
\begin{align}
o(G^g + X^1) &= \PP, \label{eq:xs724yt5rqna}\\
o(H_i^{h_i} + X^1) &= \NN \,\,\text{for any}\,\,i \in [n]. \label{eq:fy00p2s82d7o}
\end{align}
\end{lemma}

\begin{proof}[Proof of Lemma \ref{lem:distinguish2}]
Assume that $G^g \neq H_i^{h_i}$ for any $i \in [n]$.
By Lemma \ref{lem:distinguish}, for any $i \in [n]$,
there exists $Y_i^1 \in \short$ satisfying
\begin{align}
o(G^g + Y_i^1) = \NN,  \label{eq:t7ao87w6k9hw}\\
o(H_i^{h_i} + Y_i^1) = \PP. \label{eq:k64vwslgg78d}
\end{align}
Let $G^g \cong \{G'^{g'_1}_1, G'^{g'_2}_2, \ldots, G'^{g'_m}_m\}^g$,
and define $X^1 \defeq\cong \{G'^{g'_1}_1 + \es^1, G'^{g'_2}_2 + \es^1, \ldots, G'^{g'_m}_m + \es^1, Y_1^1, Y_2^1, \ldots, Y_n^1\}^1$.

We show that $X^1$ satisfies \eqref{eq:xs724yt5rqna} and \eqref{eq:fy00p2s82d7o}.
First, we confirm $o(G^g + X^1) = \PP$ because every option of $G^g + X^1$ is in $\NN$ as follows.
\begin{itemize}
\item For any $G'^{g'_i}_i \in G^g$, the option $G'^{g'_i}_i + X^1 \in G^g + X^1$ has an option $G'^{g'_i}_i + G'^{g'_i}_i  + \es^1 \in \PP$ by Theorem \ref{prop:twice}, so that $o(G'^{g'_i}_i + X^1) = \NN$.
\item For any $G'^{g'_i}_i \in G^g$, the option $G^g + (G'^{g'_i}_i + \es^1) \in G^g + X^1$ has an option $G'^{g'_i}_i + G'^{g'_i}_i  + \es^1 \in \PP$ by Theorem \ref{prop:twice}, so that $o(G^g + (G'^{g'_i}_i + \es^1)) = \NN$.
\item For any $Y_i^1 \in X^1$, the option $G^g + Y_i^1 \in G^g + X^1$ is in $\NN$ by \eqref{eq:t7ao87w6k9hw}.
\end{itemize}
Also, for any $i \in [n]$, we have $o(H_i^{h_i} + X^1) = \NN$
because $H_i^{h_i} + X^1 \owns H_i^{h_i} + Y_i^1 \in \PP$ by \eqref{eq:k64vwslgg78d}.
\end{proof}

We now give the proof of Theorem \ref{thm:distinguish} using Lemma \ref{lem:distinguish2} shown above.

\begin{proof}[Proof of Theorem \ref{thm:distinguish}]
(Proof of (a) $\implies$ (b))
By the assumption (a), there exists $X^1 \in \short$ satisfying \eqref{eq:wtv7j5xqrzec} and \eqref{eq:zig7xesatpwr}.
Choose $i, j \in [m]$ arbitrarily.
We have
\begin{align}
o(G_i^{g_i} + X^1) &= \PP, \label{eq:6541fto1cc6r}\\
o(G_j^{g_j} + X^1) &= \PP \label{eq:g3x47najh1ps}
\end{align}
by \eqref{eq:wtv7j5xqrzec}.
Hence,
\begin{align}
G'^{g'} \neq G_j^{g_j} \,\,\text{for any}\,\, G'^{g'} \in G^{g_i}_i \label{eq:gbwcfbzesnly}
\end{align}
because $x^1$ distinguishes $G'^{g'}$ from $G_j^{g_j}$:
\begin{align}
o(G'^{g'} + X^1) \eqlab{A}= \NN \neq \PP \eqlab{B}= o(G_j^{g_j} + X^1),
\end{align}
where
(A) follows from $G'^{g'} + X^1 \in G_i^{g_i} + X^1 \in \PP$ by \eqref{eq:6541fto1cc6r},
and (B) follows from \eqref{eq:g3x47najh1ps}.
By symmetry, we also have
\begin{align}
G_i^{g_i} \neq G'^{g'} \,\,\text{for any}\,\, G'^{g'} \in G^{g_j}_j. \label{eq:sk0p4mwgpfsq}
\end{align}
Therefore, we obtain $G^{g_i}_i \bowtie G^{g_j}_j$ by \eqref{eq:gbwcfbzesnly} and \eqref{eq:sk0p4mwgpfsq}.
Additionally, \eqref{eq:wtv7j5xqrzec} and \eqref{eq:zig7xesatpwr} clearly imply that $G^{g_i}_i \neq H^{h_j}_j$ for any $i \in [m], j \in [n]$.

(Proof of (b) $\implies$ (a))
Let $G_i^{g_i} \cong \{G'^{g'_{i,1}}_{i,1}, G'^{g'_{i,2}}_{i,2}, \ldots, G'^{g'_{i,p_i}}_{i, p_i}\}$ for $i \in [m]$.
By the assumption (b1), we have $G'^{g'_{i,k}}_{i,k} \neq G^{g_j}_j$ for any $G'^{g'_{i,k}}_{i,k} \in G^{g_i}_i$.
Hence, by Lemma \ref{lem:distinguish2},
for $i \in [m]$ and $k \in [p_i]$,
there exists $Y_{i,k}^1 \in \short$ such that
\begin{align}
o(G'^{g'_{i,k}}_{i,k} + Y_{i,k}^1) &= \PP, \label{eq:f85ar58892lx}\\
o(G^{g_j}_j + Y_{i,k}^1) &= \NN \,\,\text{for any}\,\,j \in [m]. \label{eq:d3gub08knex6}
\end{align}
By Lemma \ref{lem:distinguish2} and the assumption (b2),
for $i \in [n]$, there exists $Z_i^1 \in \short$ such that
\begin{align}
o(H^{h_i}_i + Z_i^1) &= \PP,  \label{eq:vwszaddq9hqq}\\
o(G^{g_j}_j + Z_i^1) &= \NN \,\,\text{for any}\,\,j \in [m]. \label{eq:fem1mjxqk13q}
\end{align}
Define $X \defeq= \left( \bigcup_{i = 1}^m\{Y_{i, 1}^1, Y_{i, 2}^1, \ldots, Y_{i, p_i}^1\} \right) \cup \{Z_1^1, Z_2^1, \ldots, Z_n^1\}$.
We show that $X^1$ satisfies \eqref{eq:wtv7j5xqrzec} and \eqref{eq:zig7xesatpwr}.
First, for any $i \in [m]$, we have $o(G^{g_i}_i + X^1) = \PP$ because every option of $G^{g_i}_i + X^1$ is in $\NN$ as follows.
\begin{itemize}
\item For any $G'^{g'_{i, k}}_{i, k} \in G^{g_i}_i$, the option $G'^{g'_{i, k}}_{i, k} + X^1 \in G^{g_i}_i + X^1$ has an option $G'^{g'_{i, k}}_{i, k} + Y_{i, k}^1 \in \PP$ by \eqref{eq:f85ar58892lx}, so that $o(G'^{g'_{i, k}}_{i, k} + X^1) = \NN$.
\item For any $Y_{j, k}^1 \in X^1$, the option $G^{g_i}_i  + Y_{j, k}^1 \in G^{g_i}_i + X^1$ is in $\NN$ by \eqref{eq:d3gub08knex6}.
\item For any $Z_j^1 \in X^1$, the option $G^{g_i}_i  + Z_j^1 \in G^{g_i}_i + X^1$ is in $\NN$ by \eqref{eq:fem1mjxqk13q}.
\end{itemize}
Also, we have $o(H^{h_i}_i + X^1) = \NN$ because $H^{h_i}_i + X^1 \owns H^{h_i}_i + Z_i^1 \in \PP$ by \eqref{eq:vwszaddq9hqq}.
\end{proof}

By applying Theorem \ref{thm:distinguish} with $m = 2$ and $n = 0$, we obtain the following corollary.

\begin{corollary}
\label{prop:link}
For any $G^g, H^h \in \short$, the following equivalence holds:
\begin{align}
G^g \bowtie H^h \iff \exists X^1 \in \short \,\,\text{s.t.}\,\, o(G^g + X^1) = o(H^h + X^1) = \PP.
\end{align}
\end{corollary}

\subsection{Simplification Theorem}
\label{subsec:simplify}

In this section, we prove Theorem \ref{prop:simplify}, which reduces a given game to an equivalent and simpler game.
To state the theorem, we first introduce the following definition.

\begin{definition}
\label{def:cover}
An option $G'^{g'}$ of $G^g \in \short$ is said to be \emph{uncovered by $H^h \in \short$}
if for any $H'^{h'} \in H^h$, it holds that $G'^{g'} \neq H'^{h'}$.
A game $G^g$ is said to be \emph{covered by} a game $H^h$ if $G^g$ has no options uncovered by $H^h$.
\end{definition}
Namely, a game $G^g$ is covered by a game $H^h$ if and only if
for any $G'^{g'} \in G^g$, there exists $H'^{h'} \in H^h$ such that $H'^{h'} = G'^{g'}$.

We now present the desired theorem that gives a sufficient and necessary condition of $G^g = H^h$, where $G^g$ is covered by $H^h$.

\begin{theorem}[Simplification Theorem]
\label{prop:simplify}
For any $G^g, H^h \in \short$, if $G^g$ is covered by $H^h$, then the following conditions (a) and (b) are equivalent.
\begin{enumerate}[(a)]
\item $G^g = H^h$.
\item The following conditions (b1) and (b2) hold.
\begin{itemize}
\item[(b1)] $\type(G^g) = \type(H^h)$.
\item[(b2)] For any option $H'^{h'}$ of $H^h$ uncovered by $G^g$, there exists $H''^{h''} \in H'^{h'}$ such that $H''^{h''} = G^g$.
\end{itemize}
\end{enumerate}
\end{theorem}

\begin{proof}[Proof of Theorem \ref{prop:simplify}]
[Proof of (b) $\implies$ (a)]
By the assumption (b1), we have $g = h$, so that $H^h \cong H^g$.
Hence, it suffices to show that $G^g$ and $H^g$ satisfy the conditions (a) and (b) of Lemma \ref{lem:equal}.
To show that, choose $X^x \in \short$ with
\begin{align}
\label{eq:7bc2h9wuvjeb}
g \lor x = 1
\end{align}
arbitrarily.

(Proof of Lemma \ref{lem:equal} (a))
Assume
\begin{align}
\label{eq:n8swoek0ownd}
o(G^g + X^x) = \PP
\end{align}
and choose $H'^{h'} \in H^g$ arbitrarily. 
We show $o(H'^{h'} + X^x) = \NN$ dividing into two cases by whether the option $H'^{h'}$ of $H^g$ is uncovered by $G^g$ or not.
\begin{itemize}
\item The case where $H'^{h'} = G'^{g'}$ for some $G'^{g'} \in G^g$: We have
\begin{align}
o(H'^{h'} + X^x) \eqlab{A}= o(G'^{g'} + X^x) \eqlab{B}= \NN,
\end{align}
where
(A) follows from $H'^{h'} = G'^{g'}$,
and (B) follows from $o(G^g + X^x) = \PP$ and \eqref{eq:7bc2h9wuvjeb}.
\item The case where the option $H'^{h'}$ of $H^g$ is uncovered by $G^g$: We confirm $H'^{h'} + X^x$ is active (i.e., $h' \lor x = 1$) dividing into three cases by $\type(G^g)$ as follows.
\begin{itemize}
\item The case $\type(G^g) = 0$: Then $h' \lor x \geq x = 0 \lor x = g \lor x = 1$ by \eqref{eq:7bc2h9wuvjeb}. 
\item The case $\type(G^g) = 1_{\exists 0}$: Then $g' = 0$ for some $G'^{g'} \in G^g$.
By \eqref{eq:7bc2h9wuvjeb} and \eqref{eq:n8swoek0ownd}, we have $o(G'^{g'} + X^x) = \NN$, so that $g' \lor x = 1$.
Hence,  we obtain $h' \lor x \geq x = 0 \lor x = g' \lor x = 1$.
\item The case $\type(G^g) = 1_{\forall 1}$: By the condition (b1), we have $\type(H^g) = 1_{\forall 1}$, which leads to $h' = 1$.
Hence, we have $h' \lor x = 1 \lor x = 1$.
\end{itemize}
Next, we show that $H'^{h'} + X^x$ has an option in $\PP$.
By the condition (b2), there exists $H''^{h''} \in H'^{h'}$ such that $H''^{h''} = G^g$, so that
\begin{align}
H'^{h'} + X^x \owns H''^{h''} + X^x \eqlab{A}= G^g + X^x \eqlab{B}\in \PP,
\end{align}
where (A) follows from $H''^{h''} = G^g$,
and (B) follows from \eqref{eq:n8swoek0ownd}.
\end{itemize}

(Proof of Lemma \ref{lem:equal} (b))
Assume that there exists $G'^{g'} \in G^g$ such that $o(G'^{g'} + X^x) = \PP$.
Since $G^g$ is covered by $H^g$, there exists $H'^{h'} \in H^g$ such that $H'^{h'} = G'^{g'}$,
so that 
\begin{align}
H^g + X^x \owns H'^{h'} + X^x \eqlab{A}= G'^{g'} + X^x \in \PP,
\end{align}
where 
(A) follows from $H'^{h'} = G'^{g'}$.

[Proof of (a) $\implies$ (b)]
The condition (b1) holds directly from Corollary \ref{cor:equiv} (ii).

Now, we show (b2) by contradiction.
Let $G^g \cong \{G'^{g'_1}_1, G'^{g'_2}_2, \ldots, G'^{g'_m}_m\}^g$.
Suppose that (b2) is false, that is, assume that there exists $H'^{h'} \defeq\cong \{H''^{h''_1}_1, H''^{h''_2}_2, \ldots, H''^{h''_n}_n\}^{h'} \in H^h$ such that
\begin{align}
H'^{h'} &\neq G'^{g'_i}_i \,\,\text{for any}\,\, i \in [m] \,\,\text{(i.e., the option}\,\, H'^{h'} \,\,\text{of}\,\,  H^h \,\,\text{is uncovered by}\,\, G^g), \label{eq:dm3sqe7x8zid} \\
H''^{h''_i}_i &\neq G^g \,\,\text{for any}\,\, i \in [n]. \label{eq:mufidizutnml}
\end{align}
By Lemma \ref{lem:distinguish} and \eqref{eq:dm3sqe7x8zid}, for $i \in [m]$, there exists $Y_i^1 \in \short$ such that
\begin{align}
o(G'^{g'_i}_i + Y_i^1) &= \PP, \label{eq:bqfbk36spskl}\\
o(H'^{h'} + Y_i^1) &= \NN. \label{eq:ack7vsbcv7cb}
\end{align}
By Lemma \ref{lem:distinguish} and \eqref{eq:mufidizutnml}, for $i \in [n]$, there exists $Z_i^1 \in \short$ such that
\begin{align}
o(H''^{h''_i}_i + Z_i^1) &= \PP, \label{eq:jrfz207xuyr0}\\
o(G^g + Z_i^1) &= \NN. \label{eq:yxn790dw5oi2}
\end{align}
Define $X^1 \defeq\cong \{Y_1^1, Y_2^1, \ldots, Y_m^1, Z_1^1, Z_2^1, \ldots, Z_n^1\}^1$.
We show that $X^1$ distinguishes $G^g$ from $H^h$.
First, we have $o(G^g + X^1) = \PP$ because every option of $G^g + X^1$ is in $\NN$ as follows.

\begin{itemize}
\item For any $G'^{g'_i}_i \in G^g$, the option $G'^{g'_i}_i + X^1 \in G^g + X^1$ has an option $G'^{g'_i}_i + Y_i^1 \in \PP$ by  \eqref{eq:bqfbk36spskl}, so that $G'^{g'_i}_i + X^1 \in \NN$.
\item For any $Y_i^1 \in X^1$, the option $G^g + Y_i^1 \in G^g + X^1$ has an option $G'^{g'_i}_i + Y_i^1 \in \PP$ by  \eqref{eq:bqfbk36spskl}, so that $G^g + Y_i^1 \in \NN$.
\item For any $Z_i^1 \in X^1$, the option $G^g + Z_i^1 \in G^g + X^1$ is in $\NN$ by \eqref{eq:yxn790dw5oi2}.
\end{itemize}
We also have $o(H^h + X^1) = \NN$ because
the option $H'^{h'} + X^1 \in H^h + X^1$ is in $\PP$ as follows.
\begin{itemize}
\item For any $H''^{h''_i}_i \in H'^{h'}$, the option $H''^{h''_i}_i + X^1 \in H'^{h'} + X^1$ has an option $H''^{h''_i}_i + Z_i^1 \in \PP$
by \eqref{eq:jrfz207xuyr0}, so that $H''^{h''_i}_i + X^1 \in \NN$.
\item For any $Y_i^1 \in X^1$, the option $H'^{h'} + Y_i^1 \in H'^{h'} + X^1$ is in $\NN$ by \eqref{eq:ack7vsbcv7cb}.
\item For any $Z_i^1 \in X^1$, the option $H'^{h'} + Z_i^1 \in H'^{h'} + X^1$ has an option $H''^{h''_i}_i + Z_i^1 \in \PP$ by \eqref{eq:jrfz207xuyr0}, so that $H'^{h'} + Z_i^1 \in \NN$.
\end{itemize}
This concludes $G^g \neq H^h$, which contradicts the assumption (a).
\end{proof}

\subsection{Canonical Form Theorem}
\label{subsec:canonical}

In this subsection, we define a unique ``canonical'' game for each equivalence class modulo $=$.

An option $G'^{g'}$ of $G^g$ is said to be reversible if
the move to $G'^{g'}$ can be reverted to some option $G''^{g''}$ of $G'^{g'}$ with $G''^{g''} = G^g$,
formally defined as follows.

\begin{definition}
\label{def:reversible}
Let $G^g \in \short$, and let $G'^{g'}$ be an option of $G^g$.
For $G''^{g''} \in G'^{g'}$, the option $G'^{g'}$ of $G^g$ is said to be \emph{reversible through $G''^{g''}$} if $G''^{g''} = G^g$.
The option $G'^{g'}$ of $G^g$ is said to be \emph{reversible} if $G'^{g'}$ is reversible through $G''^{g''}$ for some $G''^{g''} \in G'^{g'}$.
\end{definition}

We define a canonical game as a game such that it and all of its subpositions have no reversible options as follows.

\begin{definition}
\label{def:canonical}
A game $G^g \in \short$ is said to be \emph{canonical} if the following conditions (a) and (b) hold.
\begin{enumerate}[(a)]
\item Every option of $G^g$ is canonical.
\item No option of $G^g$ is reversible.
\end{enumerate}
\end{definition}

It is not trivial that each equivalence class has exactly one canonical game.
We show the existence and uniqueness of the canonical game for each equivalence class in Subsubsections \ref{subsubsec:canonical-exist}
 and \ref{subsubsec:canonical-unique}, respectively,  to prove the following theorem.

\begin{theorem}[Canonical Form Theorem]
\label{thm:canonical}
For any $G^g \in \short$, there exists a unique canonical game $H^h \in \short$ with $G^g = H^h$.
\end{theorem}

\subsubsection{The Existence of a Canonical Game}
\label{subsubsec:canonical-exist}

We prove that there exists at least one canonical game in every equivalence class.
To prove that, we first show Theorem \ref{thm:bypass}, which states that  if a game possesses a reversible option, then ``bypassing'' it yields an equivalent and simpler game.

\begin{theorem}
\label{thm:bypass}
For any $G^g, G'^{g'}, G''^{g''} \in \short$ with $G''^{g''} \in G'^g \in G^g$, if $G^g = G''^{g''}$, then $G^g = ((G\setminus \{G'^{g'}\}) \cup G'')^g$.
\end{theorem}

\begin{proof}[Proof of Theorem \ref{thm:bypass}]
Define $H^g \defeq\cong ((G\setminus \{G'^{g'}\}) \cup G'' )^g$.
We show
\begin{align}
\forall X^x \in \short, o(G^g+X^x) = o(H^g+X^x)
\end{align}
by induction on $\birth(X^x)$.
We consider the following two cases separately: the case $o(G^g + X^x) = \PP$ and the case $o(G^g + X^x) = \NN$.
\begin{itemize}
\item The case $o(G^g + X^x) = \PP$:
If $g \lor x = 0$, then clearly $o(H^g + X^x) = \PP$.
If $g \lor x = 1$, then every option of $H^g + X^x$ is in $\NN$ as follows.
\begin{itemize}
\item For any $\hat{G}^{\hat{g}} \in H \cap G''$, the option $\hat{G}^{\hat{g}} + X^{x}$ satisfies 
\begin{align}
\hat{G}^{\hat{g}} + X^{x}
\eqlab{A}\in G''^{g''} + X^x
\eqlab{B}\subseteq\NN,
\end{align}
where
(A) follows from $\hat{G}^{\hat{g}} \in G''$
and (B) follows from $o(G''^{g''} + X^x) = o(G^{g} + X^x) = \PP$ and $g'' \lor x = g \lor x = 1$ since $G''^{g''} = G^g$.
\item For any $\hat{G}^{\hat{g}} \in H \cap G$, the option $\hat{G}^{\hat{g}} + X^{x}$ satisfies 
\begin{align}
\hat{G}^{\hat{g}} + X^{x}
\eqlab{A}\in G^{g} + X^x
\eqlab{B}\subseteq\NN,
\end{align}
where
(A) follows from $\hat{G}^{\hat{g}} \in G$
and (B) follows from $o(G^{g} + X^x) = \PP$ and $g \lor x = 1$.
\item For any $X'^{x'} \in X^x$, the option $H^g + X'^{x'}$ satisfies 
\begin{align}
o(H^g + X'^{x'})
\eqlab{A}= o(G^g + X'^{x'})
\eqlab{B}= \NN,
\end{align}
where
(A) follows from the induction hypothesis,
and (B) follows from $G^g + X^x \in \PP$ and $g \lor x = 1$.
\end{itemize}

\item The case $o(G^g + X^x) = \NN$:
Then $g \lor x = 1$. Hence, it suffices to show that $H^g + X^x$ has an option in $\PP$.
Because $o(G''^{g''} + X^x) = o(G^g + X^x) = \NN$ by $G''^{g''} = G^g$, 
the game $G''^{g''} + X^x$ has an option in $\PP$.
Hence, the following two cases are possible: the case where $o(\hat{G}^{\hat{g}}+ X^x) = \PP$ for some $\hat{G}^{\hat{g}} \in G''$;
and the case where $o(G''^{g''}+ X'^{x'}) = \PP$ for some $X'^{x'} \in X^x$.
We show that $H^g + X^x$ has an option in $\PP$ for the two cases separately:

\begin{itemize}
\item The case where $o(\hat{G}^{\hat{g}}+ X^x) = \PP$ for some $\hat{G}^{\hat{g}} \in G''$:
Because $\hat{G}^{\hat{g}} \in G'' \subseteq H$, we have $H^g + X^x \owns \hat{G}^{\hat{g}} + X^x \in \PP$.

\item The case where $o(G''^{g''}+ X'^{x'}) = \PP$ for some $X'^{x'} \in X^x$:
Then 
$H^g + X^x$ has the option $H^g + X'^{x'}$, which satisfies
\begin{align}
o(H^g + X'^{x'})
\eqlab{A}= o(G^g + X'^{x'})
\eqlab{B}= o(G''^{g''} + X'^{x'})
= \PP,
\end{align}
where
(A) follows from the induction hypothesis,
and (B) follows from $G''^{g''} = G^g$.
\end{itemize}
\end{itemize}
\end{proof}

We now show the existence of a canonical game for each equivalence class as follows.

\begin{theorem}
\label{thm:canonical-exist}
For any $G^g \in \short$, there exists $H^h \in \short$ such that $H^h = G^g$ and $H^h$ is canonical.
\end{theorem}

\begin{proof}[Proof of Theorem \ref{thm:canonical-exist}]
First, we define a mapping $f \colon \short \to \uint$ as
\begin{align}
f(J^j) \defeq= 1 + \sum_{J'^{j'} \in J} f(J'^{j'}) \label{eq:kj6i1xnoj58e}
\end{align}
for $J^j \in \short$.
Note that for any $J^j, J'^{j'} \in \short$, if $J^j \owns J'^{j'}$, then
\begin{align}
f(J^{j})
\eqlab{A}= 1 + \sum_{\hat{J}^{\hat{j}} \in J} f(\hat{J}^{\hat{j}})
\eqlab{B}\geq 1 + f(J'^{j'})
> f(J'^{j'}), \label{eq:lkm94rea1yf1}
\end{align}
where
(A) follows from \eqref{eq:kj6i1xnoj58e},
and (B) follows from $J'^{j'} \in J$.

To prove the theorem, we prove the following stronger statement by induction on $f(G^g)$:
for any $G^g \in \short$, there exists $H^h \in \short$ such that $H^h = G^g$, $H^h$ is canonical, and $f(H^h) \leq f(G^g)$.

By \eqref{eq:lkm94rea1yf1}, we can apply the induction hypothesis to each option of $G^g$.
Hence, for each option $G'^{g'}$ of $G^g$, there exists $H'^{h'} \in H^h$ such that $H'^{h'} = G^{g'}$, $H^{h'}$ is canonical, and $f(H'^{h'}) \leq f(G'^{g'})$.
By replacing each option $G'^{g'}$ with the corresponding $H'^{h'}$, we may suppose that all options of $H^h$ are canonical.
If $G^g$ has no reversible options, then $G^g$ is canonical, and the proof is done.
Therefore, we may suppose that every option of $G^g$ is canonical and $G^g$ has a reversible option $G'^{g'}$ through some $G''^{g''}$.
Define $H^g \defeq\cong  ((G\setminus \{G'^{g'}\}) \cup G'')^g$.
Then we have
\begin{align}
f(H^g)
&\eqlab{A}= 1 + \sum_{H'^{h'} \in H} f(H'^{h'})\\
&= 1 + \sum_{H'^{h'} \in (G\setminus \{G'^{g'}\}) \cup G''} f(H'^{h'})\\
&\leq 1 + \sum_{H'^{h'} \in (G\setminus \{G'^{g'}\})} f(H'^{h'}) + \sum_{H'^{h'} \in G''} f(H'^{h'})\\
&= 1 + \sum_{H'^{h'} \in G} f(H'^{h'}) - f(G'^{g'}) + \sum_{H'^{h'} \in G''} f(H'^{h'})\\
&\eqlab{B}= f(G^g) - f(G'^{g'}) + \left(f(G''^{g''}) - 1\right)\\
&< f(G^g) + \left(f(G''^{g''}) - f(G'^{g'})\right)\\
&\eqlab{C}< f(G^g), \label{eq:9aezo7kvyaqx}
\end{align}
where
(A) follows from \eqref{eq:kj6i1xnoj58e},
(B) follows from \eqref{eq:kj6i1xnoj58e},
and (C) follows from \eqref{eq:lkm94rea1yf1} and $G''^{g''} \in G'^{g'}$.
Therefore, by the induction hypothesis, there exists $\hat{H}^{\hat{h}} \in \short$ such that $\hat{H}^{\hat{h}} = H^g$, $\hat{H}^{\hat{h}}$ is canonical, and $f(\hat{H}^{\hat{h}}) \leq f(H^g)$,
so that
\begin{align}
\hat{H}^{\hat{h}} = H^g \cong ((G\setminus \{G'^{g'}\}) \cup G'')^g \eqlab{A}= G^g,
\quad f(\hat{H}^h) \leq f(H^g) \eqlab{B}< f(G^g)
\end{align}
as desired, where
(A) follows from Theorem \ref{thm:bypass},
and (B) follows from \eqref{eq:9aezo7kvyaqx}.
\end{proof}

\subsubsection{The Uniqueness of the Canonical Game}
\label{subsubsec:canonical-unique}

Next, we show that every equivalence class has at most one canonical game.
We first show the following lemma.

\begin{lemma}
\label{prop:not-link}
For any $G^g, H^h \in \short$, if $G^g = H^h$, then the following statements (i) and (ii) hold,
where $\not\bowtie$ indicates that the relation $\bowtie$ does not hold.
\begin{enumerate}[(i)]
\item for any $G'^{g'} \in G^g$, we have $G'^{g'} \not\bowtie H^h$.
\item for any $H'^{h'} \in H^h$, we have $H'^{h'} \not\bowtie G^g$.
\end{enumerate}
\end{lemma}

\begin{proof}[Proof of Lemma \ref{prop:not-link}]
Assume $G^g = H^h$.
We prove this by contradiction, assuming that at least one of (i) and (ii) does not hold.
Without loss of generality, we may assume $G'^{g'} \bowtie H^h$ for some $G'^{g'} \in G^g$.
Then there exists $X^1 \in \short$ such that $o(G'^{g'} + X^1) = o(H^h + X^1) = \PP$ by Corollary \ref{prop:link}.
This game $X^1$ distinguishes $G^g$ from $H^h$:
$o(H^h + X^1) = \PP$ directly from the definition of $X^1$;
$o(G^{g} + X^1) =  \NN$ by $G^g + X^1 \owns G'^{g'} + X^1 \in \PP$.
This conflicts with $G^g = H^h$.
\end{proof}

Using Lemma \ref{prop:not-link}, we see that a canonical game $G^g$ is covered by any other equivalent game to $G^g$ as follows.

\begin{theorem}
\label{prop:canonial-include}
For any $G^g, H^h \in \short$, if $G^g$ is canonical and $G^g = H^h$, then $G^g$ is covered by $H^h$.
\end{theorem}

\begin{proof}[Proof of Theorem \ref{prop:canonial-include}]
Choose $G'^{g'} \in G^g$ arbitrarily.
Since $G^g = H^h$, we have $G'^{g'} \not\bowtie H^h$ by Lemma \ref{prop:not-link}.
Namely, we have at least one of the following statements hold: $G''^{g''} = H^h$ for some $G''^{g''} \in G'^{g'}$; or $G'^{g'} = H'^{h'}$ for some $H'^{h'} \in H^h$.
Because the former $G''^{g''} = H^h = G^g$ contradicts that $G^g$ is canonical, the latter is true.
\end{proof}

By combining Theorems \ref{prop:simplify} and \ref{prop:canonial-include}, we obtain the following corollary.

\begin{corollary}
\label{prop:canonical-simplify}
For any $G^g, H^h \in \short$, if $G^g$ is canonical, then the following conditions (a) and (b) are equivalent.
\begin{enumerate}[(a)]
\item $G^g = H^h$.
\item The following conditions (b0)--(b2) hold.
\begin{itemize}
\item[(b0)] $G^g$ is covered by $H^h$. 
\item[(b1)] $\type(G^g) = \type(H^h)$. 
\item[(b2)] For any option $H'^{h'}$ of $H^h$ uncovered by $G^g$, there exists $H''^{h''} \in H'^{h'}$ such that $H''^{h''} = G^g$.
\end{itemize}
\end{enumerate}
\end{corollary}

\begin{proof}[Proof of Corollary \ref{prop:canonical-simplify}]
(Proof of (a) $\implies$ (b))
The condition (b0) is directly from Theorem \ref{prop:canonial-include}.
The conditions (b1) and (b2) hold directly from Theorem \ref{prop:simplify}.

(Proof of (b) $\implies$ (a)) Directly from Theorem \ref{prop:simplify}.
\end{proof}

Finally, the uniqueness of the canonical game is obtained as follows.

\begin{theorem}
\label{prop:canonical-unique}
For any $G^g, H^{h} \in \short$, if $G^g$ and $H^h$ are canonical and $G^g = H^h$, then $G^g \cong H^h$.
\end{theorem}

\begin{proof}[Proof of Theorem \ref{prop:canonical-unique}]
We prove this by induction on $\birth(G^g) + \birth(H^h)$.
It suffices to show $G \subseteq H$ by symmetry.
Choose $G'^{g'} \in G^g$ arbitrarily. By Theorem \ref{prop:canonial-include}, there exists $H'^{h'} \in H^h$ such that $G'^{g'} = H'^{h'}$.
Since $G^g$ and $H^h$ are canonical, $G'^{g'}$ and $H'^{h'}$ are also canonical.
Hence, we obtain $G'^{g'} \cong H'^{h'} \in H^h$ by the induction hypothesis.
This concludes $G \subseteq H$.
\end{proof}

Combining Theorems \ref{thm:canonical-exist} and \ref{prop:canonical-unique} completes the proof of Theorem \ref{thm:canonical}.

\subsubsection{Transitive Games}

We introduce transitive games and demonstrate that they are canonical.

A game is said to be \emph{transitive} if all of its subpositions $G^g$ satisfy the following condition:
every option of every option of $G^g$ is an option of $G^g$.

\begin{definition}
\label{def:transitive}
A game $G^g\in \short$ is said to be \emph{transitive} if the following conditions (a) and (b) hold.
\begin{enumerate}[(a)]
\item Every option of $G^g$ is transitive.
\item For any $G'^{g'} \in G^g$, it holds that $G' \subseteq G$.
\end{enumerate}
\end{definition}

Transitive games are canonical as follows.

\begin{theorem}
\label{thm:trans-canonical}
For any $G^g \in \short$, if $G^g$ is transitive, then $G^g$ is canonical.
\end{theorem}

The proof of Theorem \ref{thm:trans-canonical} relies on the following lemma.

\begin{lemma}
\label{lem:trans-grundy}
For any $G^g \in \short$, if $G^g$ is transitive, then the following statements (i) and (ii) hold.
\begin{enumerate}[(i)]
\item $\{\birth(G'^{g'}) : G'^{g'} \in G^g\} = \{0, 1, 2, \ldots, \birth(G^g)-1\}$.
\item $\grundy(G^g) = \birth(G^g)$.
\end{enumerate}
\end{lemma}

\begin{proof}[Proof of Lemma \ref{lem:trans-grundy}]
(Proof of (i))
By the definition of $\birth(G^g)$, we have $\birth(G'^{g'}) < \birth(G^{g})$ for any $G'^{g'} \in G^g$,
so that the inclusion $\{\birth(G'^{g'}) : G'^{g'} \in G^g\} \subseteq \{0, 1, 2, \ldots, \birth(G^g)-1\}$ holds.
We prove the opposite inclusion $\supseteq$ by induction on $\birth(G^g)$.
By the definition of $\birth(G^g)$, there exists $\hat{G}^{\hat{g}} \in G^g$
such that $\birth(\hat{G}^{\hat{g}}) = \birth(G^g) - 1$.
We have
\begin{align}
\{\birth(G'^{g'}) : G'^{g'} \in G\}
&\eqlab{A}\supseteq \{\birth(G'^{g'}) : G'^{g'} \in \hat{G} \cup \{\hat{G}^{\hat{g}}\}\}\\
&= \{\birth(G'^{g'}) : G'^{g'} \in \hat{G}\} \cup \{\birth(\hat{G}^{\hat{g}})\}\\
&\eqlab{B}= \{0, 1, 2, \ldots, \birth(\hat{G}^{\hat{g}})-1\} \cup \{\birth(\hat{G}^{\hat{g}})\}\\
&= \{0, 1, 2, \ldots, \birth(\hat{G}^{\hat{g}})-1, \birth(\hat{G}^{\hat{g}})\}\\
&= \{0, 1, 2, \ldots, \birth(G^{g})-2, \birth(G^{g})-1\},
\end{align}
where
(A) follows from $\hat{G} \cup \{\hat{G}^{\hat{g}}\} \subseteq G$ since $G^g$ is transitive and $\hat{G}^{\hat{g}} \in G^g$,
and (B) follows from the induction hypothesis.

(Proof of (ii))
We prove this by induction on $\birth(G^g)$.
We have
\begin{align}
\grundy(G^g)
= \mex\{\grundy(G'^{g'}) : G'^{g'} \in G^g\}
\eqlab{A}= \mex\{\birth(G'^{g'}) : G'^{g'} \in G^g\}
\eqlab{B}= \mex\{0, 1, 2, \ldots, \birth(G^g)-1\}
=  \birth(G^g),
\end{align}
where
(A) follows from the induction hypothesis,
and (B) follows from (i) of this lemma.
\end{proof}

\begin{proof}[Proof of Theorem \ref{thm:trans-canonical}]
We prove this by induction on $\birth(G^g)$.
Choose $G'^{g'} \in G^g$ arbitrarily.
Since $G'^{g'}$ is transitive, it is canonical by the induction hypothesis.
Also, the option $G'^{g'}$ of $G^g$ is not reversible because
for any $G''^{g''} \in G'^{g'}$, we have
\begin{align}
\grundy(G''^{g''}) \eqlab{A}= \birth(G''^{g''}) < \birth(G'^{g'}) < \birth(G^g)  \eqlab{B}= \grundy(G^g),
\end{align}
where (A) and (B) follow from Lemma \ref{lem:trans-grundy} (ii);
this implies $G''^{g''} \neq G^g$ by Corollary \ref{cor:equiv} (iii).
\end{proof}

Note that nimbers with activeness defined in Definition \ref{def:nimber} are transitive, and there also exist countless other transitive games,
whereas the representatives of the equivalence classes of games without activeness are enumerated by $\{\st 0, \st 1, \st 2, \ldots\}$.
Moreover, there are an infinite number of canonical games that are not transitive, as follows.

\begin{theorem}
\label{thm:seq-canonical}
For any $(g_0, g_1, g_2, \ldots) \in \bin^{\infty}$ and $n \in \uint$,
the following conditions (a) and (b) are equivalent.
\begin{enumerate}[(a)]
\item $\emptyset^{g_0g_1\ldots g_n}$ is canonical.
\item $n \leq 1$; or $n \geq 2$ and $g_0g_1g_2 \in \{001, 011, 100, 101, 110\}$.
\end{enumerate}
\end{theorem}

\begin{proof}[Proof of Theorem \ref{thm:seq-canonical}]
Define $G_i^{g_i} \defeq \cong \emptyset^{g_0g_1 \ldots g_i}$ for $i \in \uint$.

(Proof of (a) $\implies$ (b))
We prove the contraposition. Namely, we assume that $n \geq 2$ and $g_0g_1g_2 \in \{000, 010, 111\}$ and 
show that $G_n^{g_n}$ is not canonical.
We prove this by induction on $n$.

For the base case $n = 2$, 
the following three cases are possible: $G_n^{g_n} \cong \es^{000}$; $G_n^{g_n} \cong \es^{010}$; and
$G_n^{g_n} \cong \es^{111}$.
By Corollary \ref{prop:canonical-simplify},
we have $G_n^{g_n} = \es^0$ in the first and second cases, and $G_n^{g_n} = \es^1$ in the third case.
Namely, in any of the three cases, the game $G_n^{g_n}$ has a reversible option, so that the game $G_n^{g_n}$ is not canonical.

For the induction step $n \geq 3$, the game $G_n^{g_n}$ is not canonical
since the option $G_{n-1}^{g_{n-1}}$ is not canonical by the induction hypothesis.

(Proof of (b) $\implies$ (a))
We prove this by induction on $n$.
\begin{itemize}
\item The case $n = 0$: The game $G_{0}^{g_0} \cong \es^{g_0}$ is clearly canonical.
\item The case $n = 1$: The only option $G_0^{g_0}$ of $G_{1}^{g_1}$ is canonical by the case $n = 0$ above, and is not reversible obviously.
\item The case $n = 2$: The only option $G_{1}^{g_1}$ is canonical by the case $n = 1$ above.
Also, the only option $G_{1}^{g_1}$ is not reversible because
 $G_{0}^{g_0} \neq G_{2}^{g_2}$ follows from Corollary \ref{cor:equiv} (ii)
 and $\type(G_{0}^{g_0}) \neq \type(G_{2}^{g_2})$ for any $g_0g_1g_2 \in \{001, 011, 100, 101, 110\}$.
 
\item The case $n \geq 3$: By induction hypothesis, the games $G_{n-3}^{g_{n-3}}$, $G_{n-2}^{g_{n-2}}$, and $G_{n-1}^{g_{n-1}}$ are canonical.
In particular, the only option $G_{n-1}^{g_{n-1}}$ of $G_n^{g_n}$ is canonical.

It remains to show that the only option $G_{n-1}^{g_{n-1}}$ of $G_n^{g_n}$ is not reversible, that is, $G_{n-2}^{g_{n-2}} \neq G_{n}^{g_{n}}$.
Since $G_{n-2}^{g_{n-2}}$ is canonical, to prove $G_{n-2}^{g_{n-2}} \neq G_{n}^{g_{n}}$, 
it suffices to show that $G^g \defeq\cong G_{n-2}^{g_{n-2}}$ and
$H^h \defeq\cong G_{n}^{g_{n}}$ does not satisfy the condition (b0) of Corollary \ref{prop:canonical-simplify}.
Since $G_{n-3}^{g_{n-3}}$ and $G_{n-1}^{g_{n-1}}$ are canonical and $G_{n-3}^{g_{n-3}} \not\cong G_{n-1}^{g_{n-1}}$,
we have $G_{n-3}^{g_{n-3}} \neq G_{n-1}^{g_{n-1}}$ by Theorem \ref{prop:canonical-unique}.
Hence, $G_{n-2}^{g_{n-2}} \cong \{G_{n-3}^{g_{n-3}}\}^{g_{n-2}}$ is not covered by $G_{n}^{g_{n}} \cong \{G_{n-1}^{g_{n-1}}\}^{g_n}$,
that is, the condition (b0) of Corollary \ref{prop:canonical-simplify} is not satisfied as desired.
\end{itemize}
\end{proof}

\begin{remark}
\label{rem:transitive}
There exist games $G^g$ and $H^h$ that satisfy $G^g \neq H^h$ and cannot be distinguished by any transitive game.
In fact, $G^g \defeq\cong \es^1, H^h \defeq\cong \es^{011}$ satisfy this condition as follows.

By Theorem \ref{thm:seq-canonical}, the games $G^g$ and $H^h$ are canonical.
Hence, Theorem \ref{prop:canonical-unique} implies $G^g \neq H^h$.

On the other hand, any transitive game $X^x$ does not distinguish $G^g$ from $H^h$ as follows.
\begin{itemize}
\item It is directly seen that $o(G^g + \es^0) = \PP =  o(H^h + \es^0)$ and $o(G^g + \es^1) = \PP = o(H^h + \es^1)$.
\item Let $X^x \in \short \setminus\{\es^0, \es^1\}$ be an arbitrary transitive game.
Since $X^x$ is transitive, $X^x \owns \emptyset^b$ for some $b \in \{0, 1\}$, so that $G^g + X^x \owns G^g + \es^b \in \PP$ and $H^h + X^x \owns H^h + \es^b \in \PP$. Also, we have $g \lor x = 1 \lor x = 1$ and $h \lor x = 1 \lor x = 1$, Therefore, $o(G^g + X^x) = \NN = o(H^h + X^x)$.
\end{itemize}
\end{remark}

\subsection{Nimbers with activeness}
\label{subsec:nimber}

In this subsection, we focus nimbers with activeness as a particular case of games with activeness,
and we see that simple rules apply to them in the case $\bm{\gamma} = \bm{0}$ or $\bm{\gamma} = \bm{1}$.

The first result states that if a game is equivalent to some nimber $\st^{\bm{\gamma}} n$, then $n$ must be equal to $\grundy(G^g)$ as follows.

\begin{theorem}
\label{thm:star-grundy}
For any $G^g \in \short$, $\bm{\gamma} \in \bin^{\infty}$, and $n \in \uint$,
if $G^g = \st^{\bm{\gamma}}n$, then $\grundy(G^g) = n$.
\end{theorem}

\begin{proof}[Proof of Theorem \ref{thm:star-grundy}]
Assume $G^g = \st^{\bm{\gamma}}n$.
We prove this by induction on $\birth(G^g)$.
The game $\st^{\bm{\gamma}} n$ is transitive and thus canonical by Theorem \ref{thm:trans-canonical}.
Hence, the assumption $G^g = \st^{\bm{\gamma}}n$ implies that the conditions (b0)--(b2) of Corollary \ref{prop:canonical-simplify} hold for $G^g$ and $\st^{\bm{\gamma}}n$.
We prove $\grundy(G^g) \geq n$ and $\grundy(G^g) \leq n$ as follows.

(Proof of $\grundy(G^g) \geq n$)
By the condition (b0) of Corollary \ref{prop:canonical-simplify}, the game $\st^{\bm{\gamma}} n$ is covered by $G^g$.
Namely, for any $n' = 0, 1, 2, \ldots, n-1$, there exists $G'^{g'} \in G^g$ such that $G'^{g'} = \st^{\bm{\gamma}}n'$,
so that $\grundy(G'^{g'}) = n'$ by the induction hypothesis. Hence, we have
$\{\grundy(G'^{g'}) : G'^{g'} \in G^g\} \supseteq \{0, 1, 2, \ldots, n-1\}$, which yields
$\grundy(G^g) = \mex\{\grundy(G'^{g'}) : G'^{g'} \in G^g\} \geq \mex\{0, 1, 2, \ldots, n-1\} = n$.

(Proof of $\grundy(G^g) \leq n$)
It suffices to show that for any $G'^{g'} \in G^g$, it holds that $\grundy(G'^{g'}) \neq n$.
We choose $G'^{g'} \in G^g$ arbitrarily and consider the following two cases separately.
\begin{itemize}
\item The case where $G'^{g'}$ is an uncovered option of $G^g$ by $\st^{\bm{\gamma}}n'$:
By the condition (b2) of Corollary \ref{prop:canonical-simplify},
there exists $G''^{g''} \in G'^{g'}$ such that $G''^{g''} = \st^{\bm{\gamma}} n$,
which leads to $\grundy(G''^{g''}) = n$ by the induction hypothesis.
Hence, we obtain
\begin{align}
\grundy(G'^{g'}) = \mex\{\grundy(\hat{G}^{\hat{g}}) : \hat{G}^{\hat{g}} \in G'^{g'}\} \eqlab{A}\neq \grundy(G''^{g''}) = n,
\end{align}
where (A) follows from $G''^{g''} \in G'^{g'}$.
\item The case where $G'^{g'}$ is not an uncovered option of $G^g$ by $\st^{\bm{\gamma}}n'$:
Then there exists $ \st^{\bm{\gamma}}n' \in \st^{\bm{\gamma}}n$ such that $G'^{g'} = \st^{\bm{\gamma}}n'$,
so that $\grundy(G'^{g'}) = n' < n$ by the induction hypothesis.
\end{itemize}
\end{proof}

The following theorem gives the mex-rule for nimbers with activeness (cf.~Proposition \ref{prop:imp-mex}).

\begin{theorem}[Mex Rule]
\label{prop:star-mex}
For any $\bm{\gamma} = (\gamma_0, \gamma_1, \gamma_2, \ldots, ) \in \bin^{\infty}$ and finite set $A \subseteq \mathbb{N}_{\geq 0}$, the following conditions (a) and (b) are equivalent, where $m \defeq= \mex A$.
\begin{enumerate}[(a)]
\item $G^{\gamma_m} \defeq\cong \{\st^{\bm{\gamma}}i : i \in A\}^{\gamma_m} = \st^{\bm{\gamma}}m$.
\item $\gamma_i = 0$ for some $i \leq m$; or $\gamma_i = 1$ for any $i \in A$.
\end{enumerate}
\end{theorem}
Note that
\begin{align}
\st^{\bm{\gamma}}m
\cong \{\st^{\bm{\gamma}}i : i \in \uint, 0 \leq i < m\}^{\gamma_m}
\eqlab{A}\subseteq \{\st^{\bm{\gamma}}i : i \in A\}^{\gamma_m}
\cong G^{\gamma_m},
\end{align}
where (A) follows from $m = \mex A$.
In particular, $\st^{\bm{\gamma}}m$ is covered by $G^{\gamma_m}$.

To prove Theorem \ref{prop:star-mex}, we first show that the condition (b) is equivalent to 
the condition $\type(G^{\gamma_m}) = \type(\st^{\bm{\gamma}}m)$ as the following lemma.

\begin{lemma}
\label{lem:star-mex}
For any $\bm{\gamma} = (\gamma_0, \gamma_1, \gamma_2, \ldots, ) \in \bin^{\infty}$ and finite set $A \subseteq \mathbb{N}_{\geq 0}$, 
the following conditions (a) and (b) are equivalent, 
where $m \defeq= \mex A$ and $G^{\gamma_m} \defeq\cong \{\st^{\bm{\gamma}}i : i \in A\}^{\gamma_m}$.
\begin{enumerate}[(a)]
\item $\gamma_i = 0$ for some $i \leq m$; or $\gamma_i = 1$ for any $i \in A$.
\item $\type(G^{\gamma_m}) = \type(\st^{\bm{\gamma}}m)$.
\end{enumerate}
\end{lemma}

\begin{proof}[Proof of Lemma \ref{lem:star-mex}]
(Proof of (a) $\implies$ (b)) 
By the assumption (a), it suffices to consider the following two cases:
the case where $\gamma_i = 0$ for some $i \leq m$, and the case where $\gamma_i = 1$ for any $i \in A$.
\begin{itemize}
\item The case where $\gamma_i = 0$ for some $i \leq m$:
We further divide into the following two cases by $\gamma_m$.
\begin{itemize}
\item The case $\gamma_m = 0$: Then $\type(G^{\gamma_m}) = 0 = \type(\st^{\bm{\gamma}}m)$ as desired.
\item The case $\gamma_m = 1$:
Then $\gamma_i = 0$ for some $i < m$.
We have $\st^{\bm{\gamma}} i \in \st^{\bm{\gamma}} m \subseteq G^{\gamma_m}$,
so that $\type(G^{\gamma_m}) = 1_{\exists 0} = \type(\st^{\bm{\gamma}}m)$.
\end{itemize}
\item The case where $\gamma_i = 1$ for any $i \in A$:
Since $\{0, 1, 2, \ldots, m-1\} \subseteq A$, we have $\gamma_0 = \gamma_1 = \cdots = \gamma_{m-1} = 1$,
so that $\type(G^{\gamma_m}) = 1_{\forall 1} = \type(\st^{\bm{\gamma}}m)$.
\end{itemize}

(Proof of (b) $\implies$ (a)) We divide into the following three cases by $\type(G^{\gamma_m})$.
\begin{itemize}
\item The case $\type(G^{\gamma_m}) = 0$:
Then $\gamma_m = 0$. In particular,  $\gamma_i = 0$ for some $i \leq m$.
\item The case $\type(G^{\gamma_m}) = 1_{\exists 0}$:
By $\type(\st^{\bm{\gamma}}m) = \type(G^{\gamma_m}) = 1_{\exists 0}$, there exists $i$ with $0 \leq i < m$ such that $\gamma_i = 0$.
\item The case $\type(G^{\gamma_m}) = 1_{\forall 1}$:
Directly from $\type(G^{\gamma_m}) = 1_{\forall 1}$, we have $\gamma_i = 1$ for any $i \in A$.
\end{itemize}
\end{proof}

\begin{proof}[Proof of Theorem \ref{prop:star-mex}]
(Proof of (a) $\implies$ (b))
Since $\st^{\bm{\gamma}}m$ is covered by $G^{\gamma_m}$,
the assumption (a) implies $\type(G^{\gamma_m}) = \type(\st^{\bm{\gamma}}m)$ by Theorem \ref{prop:simplify},
which is equivalent to the condition (b) of Theorem \ref{prop:star-mex} by Lemma \ref{lem:star-mex}.

(Proof of (b) $\implies$ (a))
It suffices to show the conditions (b1) and (b2) of Theorem \ref{prop:simplify} since $\st^{\bm{\gamma}} m$ is covered by $G^{\gamma_m}$.
The condition (b1) is directly from Lemma \ref{lem:star-mex}.
The condition (b2) is satisfied since
every $\st^{\bm{\gamma}}i \in G^{\gamma_m} \setminus \st^{\bm{\gamma}} m$ has an option $\st^{\bm{\gamma}}m$.
\end{proof}

\begin{example}
For any $\bm{\gamma} = (\gamma_0, \gamma_1, \gamma_2, \ldots) \in \bin^{\infty}$ with $\gamma_0 = 0$,
the condition (b) of Theorem \ref{prop:star-mex} is satisfied, and thus the condition (a) of Theorem \ref{prop:star-mex} also holds.
For example, if $\bm{\gamma} \defeq = (0, 1, 0, 1, 0, 1, \ldots )$, then
we have $\{\st^{\bm{\gamma}} 0, \st^{\bm{\gamma}} 1, \st^{\bm{\gamma}} 3, \st^{\bm{\gamma}} 5\}
= \st^{\bm{\gamma}} (\mex\{0, 1, 3, 5\}) =  \st^{\bm{\gamma}} 2$.
\end{example}

The next theorem corresponds to the addition rule of nimbers (cf.~Proposition \ref{prop:imp-star} (ii)).

\begin{theorem}
\label{prop:star-add}
Let $\bm{\gamma} = (\gamma_0, \gamma_1, \gamma_2, \ldots) \in \bin^{\infty}$.
Let $K \in \uint$ satisfy that
\begin{align}
\forall i, j \in \uint, \left(j < 2^K \implies \gamma_i = \gamma_{i \oplus j}\right). \label{eq:wb8dsvwfgnfv}
\end{align}
Then for any $a, b \in \mathbb{Z}_{\geq 0}$ with $b < 2^K$, we have $\st^{\bm{\gamma}} a + \st^{\bm{0}} b = \st^{\bm{\gamma}}(a\oplus b)$.
\end{theorem}

In the proof of Theorem \ref{prop:star-add}, we utilize the following properties of nim-sum.
\begin{lemma}~
\label{lem:xor}
\begin{enumerate}[(i)]
\item For any $a, b \in \uint$, we have $a \oplus b = \mex \left(\{a'\oplus b: a' < a \} \cup \{a\oplus b': b' < b\}\right)$.
\item For any $K, a, b \in \uint$, the following equivalence holds: $\lfloor a / 2^K \rfloor = \lfloor (a \oplus b) / 2^K \rfloor \iff b < 2^K$.
\end{enumerate}
\end{lemma}

We omit the proof of Lemma \ref{lem:xor}.

\begin{proof}[Proof of Theorem \ref{prop:star-add}]
We prove this by induction on $a+b$.
For the base case $a = b = 0$, we have
$\st^{\bm{\gamma}} 0 + \st^{\bm{0}} 0
\cong \st^{\bm{\gamma}} 0 + \es^0
\cong \st^{\bm{\gamma}} 0
\cong \st^{\bm{\gamma}} (0 \oplus 0)$.
We consider the induction step.
Let $A \defeq= \{a' \oplus b : a' < a\} \cup \{a \oplus b' : b' < b\}$.
Note that $\mex A = a \oplus b$ by Lemma \ref{lem:xor} (i).
We have
\begin{align*}
\st^{\bm{\gamma}} a + \st^{\bm{0}} b
&\cong (\{\st^{\bm{\gamma}} a' + \st^{\bm{0}} b : a' < a\} \cup \{\st^{\bm{\gamma}} a + \st^{\bm{0}} b' : b' < b\})^{\gamma_a \lor 0}\\
&\eqlab{A}= (\{\st^{\bm{\gamma}} (a'\oplus b) : a' < a\} \cup \{\st^{\bm{\gamma}} (a\oplus b') : b' < b\})^{\gamma_a}\\
&\cong \{\st^{\bm{\gamma}} c : c \in A\}^{\gamma_a}\\
&\eqlab{B}\cong \{\st^{\bm{\gamma}} c : c \in A\}^{\gamma_{a\oplus b}}\\
&\eqlab{C}\cong \{\st^{\bm{\gamma}} c : c \in A\}^{\gamma_{\mex A}}\\
&\eqlab{D}= \st^{\bm{\gamma}} (\mex A)\\
&\eqlab{E}\cong \st^{\bm{\gamma}} (a \oplus b),
\end{align*}
where
(A) follows from the induction hypothesis and Theorem \ref{prop:replace},
(B) follows from \eqref{eq:wb8dsvwfgnfv},
(C) follows from Lemma \ref{lem:xor} (i),
(D) follows from Theorem \ref{prop:star-mex} with justification in the next paragraph,
and (E) follows from Lemma \ref{lem:xor} (i).

To justify the application of Theorem \ref{prop:star-mex} in the equality (D) above,
we have to show that the condition (b) of Theorem \ref{prop:star-mex} is satisfied.
It suffices to show that
\begin{align}
\forall c \in A, \left(c > \mex A \implies  \gamma_c = \gamma_{\mex A}\right) \label{eq:rb6vfhx2dpvo}
\end{align}
because \eqref{eq:rb6vfhx2dpvo} implies the condition (b) of Theorem \ref{prop:star-mex} as follows.
\begin{itemize}
\item The case where $\gamma_i = 0$ for some $i \leq \mex A$: Then the condition (b) of Theorem \ref{prop:star-mex} directly holds.
\item The case where $\gamma_i = 1$ for all $i \leq \mex A$: Then \eqref{eq:rb6vfhx2dpvo} implies $\gamma_i = \gamma_{\mex A} = 1$ for any $i \in A$.
\end{itemize}
To prove \eqref{eq:rb6vfhx2dpvo}, choose $c \in A$ with $c > \mex A$ arbitrarily.
Then $c = a' \oplus b > \mex A$ for some $a' < a$; or $c = a \oplus b' > \mex A$ for some $b' < b$.
\begin{itemize}
\item The case where $c = a' \oplus b > \mex A$ for some $a' < a$:
Then
\begin{align}
\left\lfloor \frac{a'}{2^K} \right\rfloor
\leq \left\lfloor \frac{a}{2^K} \right\rfloor
\eqlab{A}= \left\lfloor \frac{a \oplus b}{2^K} \right\rfloor
\eqlab{B}= \left\lfloor \frac{\mex A}{2^K} \right\rfloor
\leq  \left\lfloor \frac{a' \oplus b}{2^K} \right\rfloor
\eqlab{C}= \left\lfloor \frac{a'}{2^K} \right\rfloor,
\end{align}
where
(A) follows from Lemma \ref{lem:xor} (ii) and $b < 2^K$,
(B) follows from Lemma \ref{lem:xor} (i),
and (C) follows from Lemma \ref{lem:xor} (ii) and $b < 2^K$.
Therefore,
\begin{align}
\left\lfloor \frac{a}{2^K} \right\rfloor
= \left\lfloor \frac{a'}{2^K} \right\rfloor
= \left\lfloor \frac{a \oplus (a \oplus a')}{2^K} \right\rfloor.
\end{align}
This leads to $a \oplus a' < 2^K$ by Lemma \ref{lem:xor} (ii), so that
\begin{align}
\gamma_c = \gamma_{a'\oplus b} \eqlab{A}= \gamma_{a'\oplus b \oplus (a \oplus a')}= \gamma_{a\oplus b} \eqlab{B}= \gamma_{\mex A}, \label{eq:t13x22lyb23f}
\end{align}
where
(A) follows from \eqref{eq:wb8dsvwfgnfv} and $a \oplus a' < 2^K$,
and (B) follows from Lemma \ref{lem:xor} (i).

\item The case where $c = a \oplus b' > \mex A$ for some $b' < b$:
We have
\begin{align}
\gamma_c = \gamma_{a\oplus b'} \eqlab{A}= \gamma_{a\oplus b' \oplus (b' \oplus b)}= \gamma_{a\oplus b} \eqlab{B}= \gamma_{\mex A},  \label{eq:u3dhv0tzdecl}
\end{align}
where
(A) follows from \eqref{eq:wb8dsvwfgnfv} and $b \oplus b' < 2^K$ since $0 \leq b' < b < 2^K$,
and (B) follows from Lemma \ref{lem:xor} (i).
\end{itemize}
\end{proof}

Regarding $\bm{\gamma} \in \{\bm{0}, \bm{1}\}$, the following simple rules hold.

\begin{theorem}~
\label{thm:star-property}
\begin{enumerate}[(i)]
\item For any $\bm{\gamma} = (\gamma_0, \gamma_1, \gamma_2, \ldots) \in \bin^{\infty}$ and $a \in \uint$, we have $\st^{\bm{\gamma}} a + \es^1 \cong \st^{\bm{1}} a$.
\item For any $G^g \in \short$, we have $G^g + \es^1 = \st^{\bm{0}}\grundy(G^g) + \es^1$.
\item For any $a, b \in \uint$, we have 
$\st^{\bm{0}} a + \st^{\bm{0}} b = \st^{\bm{0}}(a\oplus b)$, 
$\st^{\bm{0}} a + \st^{\bm{1}} b = \st^{\bm{1}}(a\oplus b)$, 
and $\st^{\bm{1}} a + \st^{\bm{1}} b = \st^{\bm{1}}(a\oplus b)$.
\end{enumerate}
\end{theorem}

\begin{proof}[Proof of Theorem \ref{thm:star-property}]
(Proof of (i))
We prove this by induction on $\birth(G^g)$.
For any $a \in \uint$, we have
\begin{align}
\st^{\bm{\gamma}}a + \es^1
\cong \{\st^{\bm{\gamma}}a'  + \es^1 : a' < a\}^{\gamma_a \lor 1}
\cong \{\st^{\bm{\gamma}}a'  + \es^1 : a' < a\}^1
\eqlab{A}\cong \{\st^{\bm{1}}a' : a' < a\}^1
\cong \st^{\bm{1}}a,
\end{align}
where
(A) follows from the induction hypothesis.

(Proof of (ii))
We prove this by induction on $\birth(G^g)$. We have
\begin{align*}
G^g + \es^1
&\cong \{G'^{g'} : G'^{g'} \in G^g\}^g + \es^1\\
&\cong \{G'^{g'} + \es^1 : G'^{g'} \in G^g\}^1\\
&\eqlab{A}= \{\st^{\bm{0}}\grundy(G'^{g'}) + \es^1: G'^{g'} \in G^g\}^1\\
&\cong \{\st^{\bm{0}}\grundy(G'^{g'}): G'^{g'} \in G^g\}^0 + \es^1\\
&\cong \{\st^{\bm{0}} a : a \in \{\grundy(G'^{g'}): G'^{g'} \in G^g\}\}^0 + \es^1\\
&\eqlab{B}= \st^{\bm{0}}(\mex\{\grundy(G'^{g'}): G'^{g'} \in G^g\}) + \es^1\\
&\cong \st^{\bm{0}} \grundy(G^g) + \es^1,
\end{align*}
where
(A) follows from the induction hypothesis,
and (B) follows from Theorem \ref{prop:star-mex}.

(Proof of (iii)) The first (resp.~second) equality is obtained by applying Theorem \ref{prop:star-add} with $\bm{\gamma} = 0$ (resp.~$\bm{\gamma} = 1$) and a sufficiently large $K$.
The third equality is seen as
\begin{align}
\st^{\bm{1}} a + \st^{\bm{1}} b
\eqlab{A}\cong (\st^{\bm{0}} a + \es^1) + \st^{\bm{1}} b
\eqlab{B}= \st^{\bm{0}} (a \oplus b) + \es^1
\eqlab{C}\cong \st^{\bm{1}} (a \oplus b),
\end{align}
where
(A) follows from (i) of this theorem,
(B) follows from the second equality of (iii) of this theorem,
and (C) follows from (i) of this theorem.
\end{proof}

\section{Non-Increasing Games}
\label{sec:mono}

In this section, we consider games, called \emph{non-increasing} games, where once a game becomes inactive, it never returns to the active status. 
We show that even when limiting the discussion to non-increasing games,
the results in the previous section hold almost unchanged.

\begin{definition}
\label{def:mono}
A game $G^g \in \short$ is said to be \emph{non-increasing} if for any $G'^{g'} \in G^g$, the following conditions (a) and (b) hold.
\begin{enumerate}[(a)]
\item $G'^{g'}$ is non-increasing.
\item $g' \leq g$.
\end{enumerate}
We define $\mono$ as the set of all non-increasing games.
\end{definition}

Note that it is possible that one of two equivalent games is non-increasing while the other is not:
$\es^0 = \es^{010}$ but $\es^0 \in \mono$ and $\es^{010} \not\in \mono$.

The next theorem guarantees that $\mono$ is closed under the addition, so that 
we can define the sum of two non-increasing games within $\mono$.

\begin{theorem}
\label{thm:sum-mono}
For any $G^g, H^h \in \mono$, we have $G^g + H^h \in \mono$.
\end{theorem}

\begin{proof}[Proof of Theorem \ref{thm:sum-mono}]
By induction on $\birth(G^g) + \birth(H^h)$.
Choose an option of $G^g + H^h$ arbitrarily.
Without loss of generality, we may assume the option is in the form $G'^{g'} + H^h$ for some $G'^{g'} \in G^g$.
We show that $G'^{g'} + H^h$ satisfies the conditions (a) and (b) of Definition \ref{def:mono} as follows.

(Proof of Definition \ref{def:mono} (a))
We have $G'^{g'} \in \mono$ by $G^g \in \mono$.
Also, we have $H^h \in \mono$ directly from the assumption.
Hence, we obtain $G'^{g'} + H^h \in \mono$ by the induction hypothesis.

(Proof of Definition \ref{def:mono} (b))
We have $g'+h \leq g+h$ since $g' \leq g$ by $G^g \in \mono$.
\end{proof}

\begin{definition}
We define a binary relation $=_{\down}$ of $\mono$ as
\begin{align}
\{(G^g, H^h) \in \mono \times \mono : \forall X^x \in \mono, o(G^g+X^x) = o(H^h+X^x)\}.
\end{align}
\end{definition}

The binary relation $=_{\down}$ is an equivalent relation of $\mono$, obviously.

To limit the discussion to non-increasing games,
for nimbers $\st^{\bm{\gamma}} n$, we consider only non-decreasing sequences $\bm{\gamma} \in \bin^{\infty}_{\leq}$,
where $\bin^{\infty}_{\leq}$ denotes the set of all non-decreasing infinite sequences over $\bin$, that is,
\begin{align}
\bin^{\infty}_{\leq} \defeq = \{\bm{\gamma} = (\gamma_0, \gamma_1, \gamma_2, \ldots) \in \bin^{\infty} : \forall i \in \uint, \gamma_i \leq \gamma_{i+1}\}.
\end{align}

Now, we replace $\short$ (resp.~$=$, $\bin^{\infty}$) with $\mono$ (resp.~$=_{\down}$, $\bin^{\infty}_{\leq}$).
Accordingly, we consider the equivalence classes modulo $=_{\down}$ instead of modulo $=$.
We see that the results in Section \ref{sec:main} hold true even if this replacement is made.
Since $G^g = H^h$ implies $G^g =_{\down} H^h$, a proof showing the relation $=$ still holds for the relation $=_{\down}$ as is.
On the other hand, to verify whether a proof showing $G^g \neq H^h$ can be directly used to show $G^g \not=_{\down} H^h$,
it is necessary to check that the game distinguishing $G^g$ from $H^h$ in the proof, referred to as a \emph{distinguisher}, is non-increasing.
The points to note are listed below.

\begin{itemize}
\item The paragraph above Theorem \ref{prop:twice}:
Since the shown counterexamples are non-increasing,
the statements (ii')--(iii') are false even in $\mono$.

\item Corollary \ref{cor:equiv}: 
The sequences $\bm{\gamma} = \bm{0}$ and $\bm{\gamma} = \bm{1}$, with which Theorem \ref{thm:equiv} is applied in the proof, are non-decreasing.

\item Theorem \ref{thm:option-neq}: 
In the proof, the distinguisher $G^g + \emptyset^1$ is non-increasing.

\item Theorem \ref{thm:distinguish}: 
In the proof, the distinguisher $X^1$ is non-increasing.

\item Lemma \ref{lem:distinguish}: 
In the proof, the distinguisher $X^1$ is non-increasing.

\item Lemma \ref{lem:distinguish2}: 
In the proof, the distinguisher $X^1$ is non-increasing.

\item Theorem \ref{thm:bypass}: 
A game obtained by bypassing is indeed non-increasing as follows.
\begin{theorem}
\label{thm:bypass-mono}
For any $G^g, G'^{g'}, G''^{g''} \in \mono$ with $G''^{g''} \in G'^{g'} \in G^g$, we have $((G\setminus \{G'^{g'}\}) \cup G'')^g \in \mono$.
\end{theorem}
\begin{proof}[Proof of Theorem \ref{thm:bypass-mono}]
Choose $H^h \in ((G\setminus \{G'^{g'}\}) \cup G'')^g$ arbitrarily.
Then at least one of $H^h \in G$ and $H^h \in G''$ holds.
We consider the following two cases separately.
\begin{itemize}
\item The case $H^h \in G$: Then $H^{h} \in \mono$ and $h \leq g$ directly from $G^g \in \mono$.
\item The case $H^h \in G''$: Then $H^{h} \in \mono$ and $h \leq g'' \leq g' \leq g$ from $G''^{g''} \in G'^g \in G^g \in \mono$.
\end{itemize}
\end{proof}

\item Theorem \ref{thm:seq-canonical}:
For any $(g_0, g_1, g_2, \ldots) \in \bin^{\infty}_{\leq}$, it is not the case $g_0g_1g_2 \in \{100, 101, 110\}$.
However, it is possible that $g_0g_1g_2 \in \{001, 011\}$.
Therefore, there are still an infinite number of canonical games that are not transitive even in modulo $=_{\down}$.

\item Remark \ref{rem:transitive}:
The shown examples, $G^g \defeq\cong \es^1, H^h \defeq\cong \es^{011}$, are non-increasing.
Therefore, even in modulo $=_{\down}$, there exist games $G^g$ and $H^h$ that satisfy $G^g \neq H^h$ and cannot be distinguished by any transitive game.

\item Theorem \ref{prop:star-mex}: 
For any $\bm{\gamma} \in \bin^{\infty}_{\leq}$, the condition (b) of Theorem \ref{prop:star-mex} always holds,
which yields the following corollary.
\begin{corollary}
\label{cor:star-mex}
For any $\bm{\gamma} = (\gamma_0, \gamma_1, \gamma_2, \ldots, ) \in \bin^{\infty}_{\leq}$ and finite set $A \subseteq \mathbb{N}_{\geq 0}$, 
we have
\begin{align}
\{\st^{\bm{\gamma}}i : i \in A\}^{\gamma_m} = \st^{\bm{\gamma}}m, \label{eq:zrh4w2v3ee35}
\end{align}
where $m \defeq= \mex A$.
\end{corollary}
Note that the equality \eqref{eq:zrh4w2v3ee35} holds true with $=$, not only with $=_{\down}$.
\end{itemize}

\bibliographystyle{plain}
\bibliography{arxiv}

\end{document}